\pdfoutput=1
\documentclass[a4paper,12pt]{article}
\usepackage[margin=2.5cm]{geometry}
\usepackage[T1]{fontenc}
\usepackage[english]{babel}
\usepackage{amsfonts, amsthm, mathrsfs, amsmath, amssymb}
\usepackage{enumerate}
\usepackage[shortlabels]{enumitem}
\usepackage{mathtools}
\usepackage{graphicx}
\usepackage{tikz}
\usepackage{tikz-cd}
\usepackage[all]{xy}
\usepackage{graphicx}
\usepackage{amsmath}
\usepackage{adjustbox}
\usepackage{graphicx}
\usepackage{extpfeil}
\usepackage{comment}
\usepackage{float}

\usepackage[mathscr]{euscript}
\usepackage[toc]{appendix}
\usepackage{hyperref}
\hypersetup{
colorlinks=true,
linkcolor=blue,
filecolor=magenta,      
urlcolor=cyan,
citecolor=[RGB]{0,204,0},
}
\usepackage[activate={true, nocompatibility}, final, tracking=true, kerning=true, spacing=true, factor=1100, stretch=10, shrink=10]{microtype}
\usepackage{authblk}
\usepackage{tikz}
\usepackage{dsfont}
\usetikzlibrary{matrix}

\usepackage{color}
\usepackage{resizegather}
\usepackage[normalem]{ulem}
\usepackage{enotez}
\setenotez{backref=true}

\usepackage{circuitikz}

\usepackage[nameinlink]{cleveref}

\usepackage{pbox}

\newcommand{\saranew}[1]{\textcolor{cyan}{#1}}

\theoremstyle{plain}
\newtheorem{thm}{Theorem}[subsection]
\newtheorem{coroll}[thm]{Corollary}
\newtheorem{lemma}[thm]{Lemma}

\newtheorem{prop}[thm]{Proposition}

\newtheorem*{thm*}{Main Theorem}

\theoremstyle{definition}
\newtheorem{defn}[thm]{Definition}
\newtheorem{remark}[thm]{Remark}
\newtheorem{obs}[thm]{Observation}
\newtheorem{notn}[thm]{Notation}

\newtheorem{example}[thm]{Example}
\newtheorem{proposal}[thm]{Proposal}
\newtheorem{hypothesis}[thm]{Hypothesis}

\Crefname{thm}{Theorem}{Theorems}
\Crefname{coroll}{Corollary}{Corollaries}
\Crefname{lemma}{Lemma}{Lemmas}
\Crefname{prop}{Proposition}{Propositions}
\Crefname{const}{Construction}{Constructions}
\Crefname{context}{Context}{Contexts}
\Crefname{conj}{Conjecture}{Conjectures}
\Crefname{algor}{Algorithm}{Algorithms}
\Crefname{exer}{Exercise}{Exercises}
\Crefname{question}{Question}{Questions}
\Crefname{defn}{Definition}{Definitions}
\Crefname{remark}{Remark}{Remarks}
\Crefname{obs}{Observation}{Observations}
\Crefname{notn}{Notation}{Notations}
\Crefname{fact}{Fact}{Facts}
\Crefname{example}{Example}{Examples}
\Crefname{thmn}{Theorem}{Theorems}
\Crefname{proposal}{Proposal}{Proposal}
\Crefname{hypothesis}{Hypothesis}{Hypothesis}
\newcommand{\bseries}[1]{ [\hspace{-0,5mm}[ {#1} ]\hspace{-0,5mm}] }

\tikzset{
  symbol/.style={
    draw=none,
    every to/.append style={
      edge node={node [sloped, allow upside down, auto=false]{$#1$}}}
  }
}
\microtypecontext{spacing=nonfrench}

\makeatletter % changes the catcode of @ to 11
\def\makeCal#1{%
\expandafter\newcommand\csname c#1\endcsname{\mathcal{#1}}}

\def\makeBB#1{%
\expandafter\newcommand\csname b#1\endcsname{\mathbb{#1}}}

\def\makeFrak#1{%
\expandafter\newcommand\csname f#1\endcsname{\mathfrak{#1}}}

\def\makeScr#1{%
\expandafter\newcommand\csname s#1\endcsname{\mathscr{#1}}}

\count@=0
\loop
\advance\count@ 1
\edef\y{\@Alph\count@} % Alph is upper case 
\expandafter\makeCal\y
\expandafter\makeBB\y
\expandafter\makeFrak\y
\expandafter\makeScr\y
\ifnum\count@<26
\repeat
\def\makelowercaseFrak#1{%
\expandafter\newcommand\csname mf#1\endcsname{\mathfrak{#1}}}
\count@=0
\loop
\advance\count@ 1
\edef\y{\@alph\count@} % alph is lower case.. 
\expandafter\makelowercaseFrak\y
\ifnum\count@<26
\repeat
\DeclareMathOperator{\GL}{GL}

\DeclareMathOperator{\AlgSt}{\mathscr{G}}

\newcommand{\spec}{\operatorname{Spec}}

\DeclareMathOperator{\NS}{NS}

\DeclareMathOperator{\ord}{ord}

\DeclareMathOperator{\type}{type}

\DeclareMathOperator{\colim}{colim}

\DeclareMathOperator{\Map}{\operatorname{Map}}

\DeclareMathOperator{\SL}{SL}

\newcommand{\B}{\mathcal{B}}

\newcommand{\on}{\operatorname}

\newcommand{\st}{\on{st}}
\newcommand{\rep}{\on{rep}}

\newcommand{\pre}{\on{pre}}

\newcommand{\dual}{\vee}

\newcommand{\Proj}{\on{Proj}}

\newcommand{\Aut}{ \on{Aut} } 

\newcommand{\Spec}{\on{Spec}}

\newcommand{\Pic}{\on{Pic}}

\usepackage{tikz}

\newcommand*{\triple}[2][.1ex]{%
  \mathrel{\vcenter{\offinterlineskip%
  \hbox{$#2$}\vskip#1\hbox{$#2$}\vskip#1\hbox{$#2$}}}}

\newcommand\dirac{/\kern-1.2ex\partial} % Dirac operator
\newcommand\qu{/\kern-.7ex/} % Categorical quotients
\newcommand\lqu{\backslash \kern-.7ex \backslash} % Categorical
\newcommand\dr{r_+ \kern-.7ex - \kern-.7ex r_-}
 
\newcommand{\pertgen}{\mathrm{P}\overline{\cM}}
\newcommand{\bigpert}{\mathscr{A}nc}
\newcommand{\sh}{\mathrm{sh}}

\tikzset{
  symbol/.style={
    draw=none,
    every to/.append style={
      edge node={node [sloped, allow upside down, auto=false]{$#1$}}}
  }
}

    \newtheoremstyle{TheoremNum}
        {\topsep}{\topsep}              %%% space between body and thm
        {\itshape}                      %%% Thm body font
        {}                              %%% Indent amount (empty = no indent)
        {\bfseries}                     %%% Thm head font
        {.}                             %%% Punctuation after thm head
        { }                             %%% Space after thm head
        {\thmname{#1}\thmnote{ \bfseries #3}}%%% Thm head spec
    \theoremstyle{TheoremNum}

\begin{document}

\title{\textbf{Alternative moduli of twisted curves}}

\author{Daniel Halpern-Leistner, Andres Fernandez Herrero,\linebreak Giovanni Inchiostro, Sara Stephens}

\date{}

\maketitle

\abstract{We define new Artin stack compactifications of the moduli of smooth curves, which parameterize ``relatively polarized curves'' consisting of stacky semistable curves equipped with line bundles that are trivialized away from the rational bridges. Using techniques from birational geometry, we prove that stacks of relatively polarized curves admit the classical coarse space $\overline{M}_g$ of stable curves as a good moduli space. One of our main motivations is the companion paper \cite{paper_2}, where we construct a proper moduli space of $\mathrm{SL}_N$-bundles on relatively polarized curves. This suggests the hypothesis that relatively polarized curves are suitable sources for defining Gromov--Witten theory for Artin stack targets.}

\tableofcontents

%Hide subsections from table of contents in the intro
\addtocontents{toc}{\protect\setcounter{tocdepth}{1}}

\section{Introduction}

The first, and most fundamental, moduli problem studied in algebraic geometry was the moduli of smooth curves of genus $g$. In order to study the topology of that moduli space, one often compactifies it by adding nodal curves as in the Deligne-Mumford compactification $\overline{M}_g$ \cite{deligne_mumford}. Likewise, in enumerative geometry, one compactifies the moduli of smooth curves embedded in a smooth and proper target variety in order to reduce enumerative questions to questions about intersection theory on moduli spaces. In Gromov--Witten theory this compactification, Kontsevich's moduli of stable maps \cite{kontsevich}, again allows smooth curves to degenerate to nodal curves.

The problem of extending the theory of stable maps to define Gromov--Witten invariants for \emph{stacks} has been an organizing question in the subject for some time. There is a very nice answer for smooth and proper Deligne--Mumford (DM) stacks, which involves allowing twisted curves -- nodal curves with $\mu_r$ orbifold structure allowed at the nodes and marked points \cite{abramovich_vistoli_compactifying_stable_maps}. There is also a further generalization to tame algebraic stacks in positive characteristic \cite{twisted_stable_maps}.

For non-orbifold quotients, and specifically for quotients $X/G$ of affine schemes by reductive groups, the most developed theory is the theory of quasi-maps \cite{stablequasimapsgitquotients}. The theory applies to a Neron-Severi class $\ell \in \NS(X/G)_{\bQ}$ for which $G$ acts with finite stabilizers on the $\ell$-semistable locus, $X^{\ell-\rm{ss}} \subset X$, in the sense of geometric invariant theory (GIT). For such an $\ell$, the space of $\ell$-stable quasi-maps parameterizes maps from nodal curves to $X/G$ that land in $X^{\ell-\rm{ss}}$ away from finitely many base points and satisfy a certain stability condition. For a sufficiently large positive scaling of $\ell$, it agrees with the space of stable maps to the GIT quotient orbifold $X^{\ell-\rm{ss}}/G$. Studying wall-crossing, i.e., how these moduli spaces depend on $\ell$, has been significantly complicated by the breakdown of the theory for $\ell$ such that points of $X^{\ell-\rm{ss}}$ have positive dimensional stabilizers.

%The existence of different stability conditions raises the question of wall crossing, i.e., how the moduli spaces change as $\ell$ varies. The wall-crossing technique applied to the change in quasi-maps spaces as one scales $\ell$, which is typically referred to as $\epsilon$-wall-crossing, has been very useful for computing Gromov-Witten invariants of manifolds and orbifolds arising as GIT quotients, because spaces of quasi-maps are often simpler. However, quasi-map wall crossing when $\ell$ crosses a wall in the GIT fan for $X/G$ is less well understood, because that necessitates crossing through an $\ell$ for which point in $X^{\ell-\rm{ss}}$ have positive dimensional stabilizers.

Naively speaking, the ideal answer to this circle of questions would be to identify some very large open substack of the stack $\cM^{\pre}(X/G)$ of nodal curves with a map to $X/G$, and to show that different choices of stability parameters determine a $\Theta$-stratification \cite{halpernleistner2022structure} of this large stack whose semistable locus admits a proper good moduli space, including the spaces of quasi-maps as special cases. Then one could use this stratification to study wall-crossing. However, this structure has not been identified even for the target $BG$, despite the question receiving attention from many perspectives, and progress for $G=\GL_N$ \cite{gieseker-degeneration, Caporaso, nagaraj-seshadri-2, schmitt-hilbert-compactification, kausz_stable_maps, nodal_ns_moduli}.

This paper revisits the foundational assumption that one should focus on open substacks of $\cM^{\pre}(X/G)$. Perhaps instead of plain nodal (twisted) curves, one should consider nodal curves along with some additional structure, which is vacuous when the curve is stable and which has good moduli-theoretic properties. Recall that a twisted curve is called \emph{semistable} if the map to its stabilization contracts only chains of rational curves, called rational bridges, over nodes.

\begin{defn}[Anchored curves]\label{D:rigid_polarized_curves} \label{defn: stack of general curves with line bundles} \label{defn: stack of relatively polarized twisted curves} \label{defn: big stack of perturbed curves}
    The stack of anchored curves $\bigpert_{g,n}$ sends a $k$-scheme $T$ to the groupoid of tuples $(\mathscr{C},\Sigma_{\bullet}, \cL, s: \cL \to \cO_{\mathscr{C}})$, where
    \begin{enumerate}[(1)]
        \item $(\mathscr{C}, \Sigma_{\bullet}) \to T$ is a family of $n$-marked semistable twisted curves of genus $g \ge 2$ as in \Cref{defn:marked-twisted-curve}, with coarse space $(C, \sigma_{\bullet}) \to T$,
        \item $\cL$ is an invertible sheaf on $\mathscr{C}$, and
        \item $s: \cL \to \cO_{\mathscr{C}}$ is a cosection.
    \end{enumerate}
    In addition, this data is required to satisfy the following conditions:
    \begin{enumerate}[(a)]
        \item The stabilizers of $\mathscr{C}$ act faithfully on the fibers of $\cL$, i.e., $\mathscr{C} \rightarrow B\bG_m$ is representable;
        \item $\cL$ is ample relative to the stabilization morphism $(\mathscr{C},\Sigma_\bullet) \to (C^{\st}, \sigma^{\st}_{\bullet})$, meaning that Zariski locally on $T$, there exists some positive integer $N>0$ such that $\cL^{\otimes N}=: L$ descends to a line bundle $L$ on $C$ that is ample relative to $C^{\st}$;
        \item For all $t \in T$ the base-change $s_t: \cL_t \to \cO_{\mathscr{C}_t}$ is an isomorphism at every generic point of $C^{\st}_t$, thought of as a point in $\mathscr{C}_t$.    
    \end{enumerate}
We will let $\mathscr{X}_{g,n}$ denote the stack whose $T$-points consist of only the data (1) and (2) satisfying conditions (a) and (b). We will use the notation $\bigpert_g := \bigpert_{g,0}$.
\end{defn}

We define the \emph{degree} of a line bundle on a twisted curve $\mathscr{C}$ in terms of the Euler characteristic $\chi(-)$,
\[
\deg(\cL) := \frac{1}{N}\left(\chi(\cL^{\otimes N}) - \chi(\cO_{\mathscr{C}}) \right),
\]
where $N$ is any common multiple of the orders of all automorphism groups of $\sC$. For such an $N$, $\cL^{\otimes N}$ descends to the coarse moduli space of $\sC$, and $\chi(\cL^{\otimes N})-\chi(\cO_{\mathscr{C}})$ is the usual degree of the descended bundle. Note that $\deg(\cL)$ is a locally constant function on $\bigpert_g$. We say that an anchored curve $(\mathscr{C}, \cL, s)$ has degree $0$ if $\cL$ has degree $0$.

If $(\sC,\cL,s)$ is a point of $\bigpert_g$ such that the coarse moduli space is stable, then condition (c) above implies that the degree of $\cL$ is non-positive on every component. Therefore, if $\deg(\cL) = 0$, it must be that $s : \cL \to \cO_{\sC}$ is an isomorphism, and then condition (a) implies that the automorphism groups of $\sC$ are trivial. So, the additional data of $\cL$, $s$, and possible orbifold structure are only visible on non-stable curves.

\Cref{D:rigid_polarized_curves} is inspired by the beyond GIT program \cite{halpernleistner2022structure, alper2023existence, guide_beyond_git}, which reduces the construction of a good moduli space to the conditions of $\Theta$- and $S$-completeness. These conditions involve certain regular $2$-dimensional schemes $Y$ with a $\bG_m$-action and a codimension-$2$ fixed point $\mathfrak{o} \in Y$, where the condition is that a morphism $(Y \setminus \mathfrak{o}) / \bG_m \to \cX$ extends uniquely to $Y/\bG_m$. The issue with $\cX = \cM^{\pre}(X/G)$ is that even if an extension exists, it might not be unique.

This is really a pathology with the stack of nodal curves, though. We will see many examples below of families of nodal curves over $Y\setminus \mathfrak{o}$ that extend non-uniquely to families over $Y$. This non-uniqueness is closely related to the failure of the diagonal to be affine (see \Cref{remark: affine diagonal reductions}), which is in turn related to the lack of a canonical ample bundle on a nodal curve. Our \Cref{D:rigid_polarized_curves} addresses this pathology with the stack of all nodal curves. %That is why we include the ample line bundle $\cL$ as part of the data -- there is no canonical way to associate an ample line bundle to a map $C \to X/G$.

\begin{thm}[= \Cref{T:s-complete-and-theta-red}]
    The stacks $\bigpert_{g,n}$ and $\mathscr{X}_{g,n}$ have affine diagonal, and they are $\Theta$- and $S$-complete.
\end{thm}

This result inspires the following:

\begin{proposal}\label{PR:main}
    Gromov--Witten style moduli problems for the stacks $X/G$, with $X$ affine and $G$ reductive, should be constructed relative to the stack of degree $0$ anchored curves.
\end{proposal}

In a companion paper to this one \cite{paper_2}, we will realize this proposal when $X/G = B\SL_N$. We enlarge the stack of $\SL_N$ bundles on the universal curve over $\cM_g$ by considering the stack of degree $0$ anchored curves along with an $\SL_N$ bundle satisfying certain open conditions. The enlargement has several nice properties, including a $\Theta$-stratification whose semistable locus admits a proper and flat good moduli space over $\overline{M}_g$.

\bigskip

In the second part of the paper, we study a smaller open substack $\pertgen_g \subset \bigpert_g$. The motivation is that for a degree $0$ point $(\cC,\cL,s)$ of $\bigpert_g$, if the degree of $\cL$ on the rational bridges of $\sC$ is large, then the Cartier divisor $s^\vee :\cO_{\sC} \to \cL^\dual$ on the stable components of $\cC$ can deform inside the corresponding fiber over the stack $\mathfrak{M}^{tw}_g$ of twisted curves. On the other hand, if $(\sC,\cL,s)$ is a point of $\pertgen_g$, then $s^\vee$ is forced to vanish only on the rational bridges \Cref{L:perturbed_curve_basic_properties}, and the divisor on the stable components of $\sC$ is rigid in the fiber over $\mathfrak{M}^{tw}_g$.

By an \emph{orientation} of a rational bridge $\sR \subset \sC$ in a twisted curve, we mean a choice of labels $(p_-,p_+)$ for the two points on the stable components of $\sC$ where $\sR$ meets the rest of the curve. Given an orientation of $\sR$, we let $\ord_{\pm}(\sR)$ denote the order of vanishing of the section $s^\vee$ at $p_{\pm}$ on the stable components, divided by the order of the stabilizer at that point.

\begin{defn}[Relatively polarized curves]\label{D:pertgen}
    The substack $\pertgen_g \subset \bigpert_g$ of \emph{relatively polarized curves} consists of $T$-families $(\sC, \cL, s)$ of degree $0$ anchored curves such that for every geometric point $t \to T$, and for every rational bridge $\sR \subset \sC_t$, with some choice of orientation, one has
    \[
    \ord_{\pm}(\sR) \leq 1 \quad \text{and} \quad \ord_+(\sR) + \ord_-(\sR) = \deg(\cL|_{\sR}).
    \]
\end{defn}

This notion fell out of our work \cite{paper_2} on the stack of $\SL_N$ bundles on degree $0$ anchored curves, when we noticed that the stack we study there lives over the smaller substack $\pertgen_g \subset \bigpert_g$. We will discuss the connection more in \cite{paper_2}

In this paper, we focus on a separate interesting phenomenon, first identified in \cite{ascher2023moduli_CY}: $\pertgen_g$ admits an exhaustive ascending family of open substacks that all have the same good moduli space, $\overline{M}_g$. These substacks are parameterized by the following datum.

\begin{hypothesis}\label{H:stability_condition}
Let $\cB$ be a collection of finite subsets
\[
\cB \subset \{S \subset [-1,1] \cap \bQ \mid |S|<\infty \text{ and } \max(S)>0>\min(S)\}
\]
that satisfies the following conditions:
\begin{enumerate}[(1)]
    \item $S \in \cB \Rightarrow -S \in \cB$
    \item $S,T \in \cB \Rightarrow S \cup T \in \cB$
    \item If $S \in \cB$ and $T \subset S$ with $\max(T)>0>\min(T)$, then $T \in \cB$.
\end{enumerate}
\end{hypothesis}

To any oriented rational bridge $\sR$ in a degree $0$ anchored curve, in \Cref{D:type} we associate a finite subset $\type(\sR) \subset \bQ \cap [-1,1]$ with $\max(\type(\sR))>0>\min(\type(\sR))$. See also \Cref{L:type_determines_isomorphism_class}(2) for an explicit description of the relation between $\type(\sR)$ and the isomorphism class of the pair $(\sR, \cL|_{\sR})$.

\begin{defn}\label{D:stable_rigid_polarized_curves}
    For any $\cB$ as in \Cref{H:stability_condition}, let $\pertgen_g^\cB \subset \pertgen_g$ denote the substack of points $(\sC,\cL,s)$ such that for every rational bridge $\sR \subset \sC$, $\type(\sR) \in \cB$.
\end{defn}

When $\cB = \emptyset$, then $\pertgen_g^{\cB} = \overline{\cM}_g$, and when $\cB$ is the collection of all finite subsets, $\pertgen_g^\cB = \pertgen_g$. Another way to produce such an $\cB$ is to choose an arbitrary subset $X \subset [-1,1] \cap \bQ$ such that $-X = X$, and let $\cB$ be the collection of all finite subsets $S \subset X$ with $\max(S)>0>\min(S)$.

\begin{thm}\label{T:perturbed_curves_main}
    For any $\cB$ as in \Cref{H:stability_condition}, $\pertgen_g^\cB \subset \bigpert_g$ is an open substack that is $\Theta$- and $S$-complete. If $\cB$ is finite, then $\pertgen_g^\cB$ is bounded and $\pertgen_g^\cB \to \overline{M}_g$ is a good moduli space.
\end{thm}

\subsection*{Notation and conventions}
We work over a field $k$ of characteristic $0$. Unless otherwise stated, all of our schemes and algebraic stacks are implicitly assumed to be locally Noetherian. We use the term DVR to refer to a discrete valuation ring. Given a local ring $A$, we will denote by $A^{\sh}$ its strict henselization. For an algebraic space $X$ equipped with an action of a group scheme $G$, we denote the quotient stack by $X/G$ (as opposed to $[X/G]$).

Throughout the paper, we fix an integer $g \geq 2$, which will be the genus of the curves in our moduli problems.

\subsection*{Acknowledgements} We would like to thank Chiara Damiolini, Stefano Filipazzi, Roberto Fringuelli, Jochen Heinloth, Jacob Keller, Filipo Viviani, and Junyan Zhao for useful conversations related to the contents of this paper. SS was partially supported by an NSERC PGS-D scholarship. GI was partially supported by NSF grant DMS-2502104. DHL was partially supported by NSF CAREER grant DMS-1945478. DHL used Chat GPT 5 and Astra to prove \Cref{L: nef-algorithm} and to work through and generalize the main result of \cite{hall2023mayer} in \Cref{appendix}.

\addtocontents{toc}{\protect\setcounter{tocdepth}{2}}

\section{Anchored curves}
In this section we define the stack $\bigpert_{g,n}$ of anchored curves, see \Cref{defn: big stack of perturbed curves}. Roughly speaking, $\bigpert_{g,n}$ parameterizes $n$-marked twisted semistable curves $\mathscr{C}$ of genus $g$ equipped with a line bundle $\cL$ that is ample relative to the stabilization and is trivialized generically on each irreducible component of $\mathscr{C}$ that is not contracted by the stabilization morphism. The main theorem of this section is \Cref{T:s-complete-and-theta-red}, where we prove that the stack $\bigpert_{g,n}$ satisfies the $\Theta$-completeness and $S$-completeness conditions introduced in \cite{alper2023existence}.

\subsection{The Moduli of Anchored Curves}

\subsubsection*{Preliminaries on Twisted Curves}

We denote by $\mathfrak{M}_{g, n}$ the smooth stack of prestable marked curves of genus $g$ with $n$ marked points. For a marked prestable curve $(C, (\sigma_i)_{i=1}^n)$ defined over a field, we denote by $\omega_C$ the canonical line bundle of $C$. We call $\omega_C(\sum_{i=1}^n \sigma_i)$ the log canonical bundle. We consider two open substacks of $\mathfrak{M}_{g,n}$:
\begin{itemize}
    \item The substack $\overline{\cM}_{g, n}$ of stable curves; curves whose log canonical bundle is ample.
    \item The substack $\mathfrak{M}_{g, n}^{ss}$ of semistable curves; curves whose log canonical bundle is nef.
\end{itemize}

The stack of stable curves $\overline{\cM}_{g,n}$ is known to be smooth, proper, Deligne--Mumford and irreducible. The open inclusion $\overline{\cM}_{g,n} \hookrightarrow \mathfrak{M}_{g,n}$ admits a retract $\st: \mathfrak{M}_{g,n} \to \overline{\cM}_{g,n}$ called the stabilization morphism, which sends a family of marked curves $(C, (\sigma_i)_{i=1}^n)$ to its stabilization $(C^{\st}, (\sigma_i^{\st})_{i=1}^n)$. The morphism $\st: \mathfrak{M}_{g,n} \to \overline{\cM}_{g,n}$ of smooth stacks is syntomic by \cite[Prop. 3]{behrend-gromov-witten-theory}, and is of relative dimension $0$ by standard deformation theory of nodal curves. 

Given a prestable marked curve $(C, \sigma_{\bullet}) \in \mathfrak{M}_{g,n}$ defined over an algebraically closed field, recall that a maximal connected subcurve $R \subset C$ contracted by the stabilization morphism $(C, \sigma_{\bullet}) \to (C^{\st}, \sigma_{\bullet}^{\st})$ is called a rational bridge if it is an unmarked chain of rational curves such that the intersection $R \cap \overline{C \setminus R}$ consists of exactly two closed points of $C$ (cf. \cite[\href{https://stacks.math.columbia.edu/tag/0E3M}{Tag 0E3M}]{stacks-project}). We call these two points the end points of the bridge. On the other hand, such a maximal connected subcurve $R \subset C$ is called a rational tail if it is a chain of rational curves such that $R \cap \overline{C \setminus R}$ consits of one point, which we call the attaching point of the tail (cf. \cite[\href{https://stacks.math.columbia.edu/tag/0E3G}{Tag 0E3G}]{stacks-project}). If $R \subset C$ is a rational tail, then there is a unique irreducible component $\Sigma \subset R$ such that $\omega_C|_{\Sigma}$ is negative; we call this the terminal component of the tail.

A marked curve $(C, \sigma_{\bullet}) \in \mathfrak{M}_{g,n}$ defined over an algebraically closed field is semistable if and only if any maximal connected reduced subcurve $R \subset C$ contracted by the stabilization morphism $C \to C^{\st}$ is either:
\begin{itemize}
    \item a rational bridge (with no markings), or
    \item a rational tail that contains exactly one marking located in its terminal component.
\end{itemize} 

Recall the following notion of a $n$-marked twisted curve as in \cite[Def. 1.2]{olsson-twisted-curve}.

\begin{defn} \label{defn:marked-twisted-curve}
An $n$-marked twisted curve of genus $g$ over a $k$-scheme $T$ is a morphism $f: \mathscr{C} \to T$ together with a
tuple of disjoint closed substacks $(\Sigma_i)_{i=1}^n$ of $\mathscr{C}$ such that:
\begin{enumerate}[(1)]
\item $\mathscr{C}$ is a proper, finitely presented, with connected one-dimensional geometric fibers.  Its coarse moduli space $\mathscr{C} \to C$ is a nodal curve of genus $g$ over $T$.
\item Each $\Sigma_i$ is an \'etale gerbe over T, and
it's image in $C$ is contained in the smooth locus of
$C \to T$.
\item If $\mathcal{U} \subset \mathscr{C}$ denotes the complement of the $\Sigma_i$ and the singular locus of $\mathscr{C} \to T$, then $\mathcal{U} \to C$ is an open immersion.
\item For every geometric point $\bar t \to T$ whose image lies on a marking, there is an integer $r \ge 1$ such that \'etale locally, the fiber $\mathscr{C}_{\bar t}$ near that marking is isomorphic to
\[
\Spec(k(\bar t)[z])/\mu_r,
\]
where $\zeta \in \mu_r$ acts by $z \mapsto \zeta z$,
\item For every geometric point $\bar t \to T$ whose image is a node, there is an integer $r \ge 1$ such that \'etale locally, $\mathscr{C}_{\bar t}$ near that node is isomorphic to
\[
\Spec(k(\bar t)[x,y]/(xy))/\mu_r,
\] where $\zeta \in \mu_r$ acts by $ (x, y) \mapsto (\zeta x, \zeta^{-1} y)$.
\end{enumerate}
\end{defn}
We denote by $\mathfrak{M}^{tw}_{g, n}$ the smooth algebraic stack parameterizing families of $n$-marked twisted curves of genus $g$ \cite[Thm. 1.10]{olsson-twisted-curve}. The open immersion $\mathfrak{M}_{g, n} \hookrightarrow \mathfrak{M}^{tw}_{g, n}$ admits a retract $c: \mathfrak{M}^{tw}_{g, n} \to \mathfrak{M}_{g, n}$ given by sending a family of marked twisted curves to its coarse space. It follows from the local description in \cite[\S5]{olsson-twisted-curve} that the morphism $c$ is syntomic of relative dimension $0$. We denote by $\mathfrak{M}^{tw,ss}_{g, n} := c^{-1}(\mathfrak{M}^{ss}_{g, n}) \subset \mathfrak{M}^{tw}_{g, n}$ the open substack of marked twisted curves whose coarse space is semistable. We call families in $\mathfrak{M}^{tw,ss}_{g, n}$ marked twisted semistable curves. The stabilization morphism for twisted semistable curves is the composition $\st: \mathfrak{M}_{g, n}^{tw,ss} \xrightarrow{c} \mathfrak{M}_{g, n}^{ss} \xrightarrow{\st} \overline{\cM}_{g, n}$, which is syntomic of relative dimension $0$. Given a marked twisted semistable curve over an algebraically closed field $\mathscr{C}$,  a reduced closed substack $\mathscr{R} \subset \mathscr{C}$ is called a rational bridge (resp. rational tail) if its image under the coarse space morphism $\mathscr{C} \to C$ is a rational bridge (resp. rational tail). We say that $\mathscr{C}$ has $\ell$ rational bridges (resp. tails) if its coarse space $C$ does. Given a rational bridge $\mathscr{R} \subset \mathscr{C}$, its end points are defined to be the preimages of the end points of the corresponding rational bridge $R \subset C$ in the coarse space. Similarly, we define the attaching point of a rational tail $\mathscr{R} \subset \mathscr{C}$ as the preimage of the attaching point in the coarse space.

\begin{remark} \label{remark: duality}
    Given a twisted $n$-marked semistable curve $(\mathscr{C}, \Sigma_{\bullet}) \to S$ over a scheme $S$, we will usually denote by $\varphi: (\mathscr{C}, \Sigma_{\bullet}) \to (C, \sigma_{\bullet})$ its coarse space and by $\pi: (C, \sigma_{\bullet}) \to (C^{\st}, \sigma_{\bullet}^{\st})$ the stabilization of the coarse space. If the number of markings is $n=0$, then we have identifications of canonical bundles $\omega_{\mathscr{C}/S} \cong \varphi^*(\omega_{C/S}) = \varphi^*\pi^*(\omega_{C^{\st}/S})$, where the first equality follows from \cite[Prop. 2.5.1]{chiodo-twisted-curves} and the second equality follows from the curve $C$ being semistable. In particular, Grothendieck duality implies that for any given vector bundles $\cE$ and $\cF$ on $C^{\st}$ and $\mathscr{C}$ respectively, there is a canonical correspondence between morphisms $\cF \to \varphi^*\pi^*(\cE)$ and morphisms $\pi_*\varphi_*(\cF) \to \cE$. We shall use this fact at several points in the paper.
\end{remark}

\subsubsection*{The Moduli of Anchored Curves}
Recall the stacks $\bigpert_{g,n}$ and $\mathscr{X}_{g,n}$ from \Cref{defn: big stack of perturbed curves}.

\begin{prop}\label{p:forget-B-to-X} \label{prop: big stack of perturbed curves is algebraic}
    $\mathscr{X}_{g,n}$ and $\bigpert_{g,n}$ are algebraic stacks locally of finite type and with affine diagonal. The forgetful morphism
    \[\bigpert_{g,n} \to \mathscr{X}_{g,n}, \; \; \; (\mathscr{C},\Sigma_{\bullet}, \cL, s: \cL \to \cO_{\mathscr{C}}) \mapsto (\mathscr{C},\Sigma_{\bullet}, \cL) \]
    is quasiaffine and of finite type.
\end{prop}
\begin{proof}
     Let us first check that $\mathscr{X}_{g,n}$ is an algebraic stack locally of finite type over $k$. If we denote by $(\mathscr{C}_{univ}, \Sigma_{\bullet}) \to \Map_{\mathfrak{M}^{tw,ss}_g}$ the universal family, then the mapping stack $\Map_{\mathfrak{M}^{tw,ss}_g}(\mathscr{C}_{univ}, B\mathbb{G}_m)$ is algebraic and locally of finite type over $k$ by \cite[Thm. 1.2]{hall-rydh-tannaka}. Note that $\Map_{\mathfrak{M}^{tw,ss}_g}(\mathscr{C}_{univ}, B\mathbb{G}_m)$ parameterizes pairs of a twisted semistable curve $(\mathscr{C}, \Sigma_{\bullet})$ and a line bundle $\cL$ on $\mathscr{C}$. The inclusion $\mathscr{X}_{g,n} \hookrightarrow \Map_{\mathfrak{M}^{tw,ss}_g}(\mathscr{C}_{univ}, B\mathbb{G}_m)$ is an open immersion, because of the following:
     \begin{itemize}
     \item The condition on $\cL$ being ample relative to the stabilization is open.
      \item The condition on the representability of the morphism $f: \mathscr{C} \to B\mathbb{G}_m$ corresponding to $\cL$ is open. Indeed, the relative inertia $I_{f}$ is a closed substack of the inertia stack $I_{\mathscr{C}}$, which is finite and representable over $\mathscr{C}$. The identity section $\mathscr{C} \to I_f$ induces surjection $\cO_{I_f} \twoheadrightarrow \cO_{\mathscr{C}}$ of coherent sheaves on $\mathscr{C}$. Let $\cK$ denote the kernel of this surjection; it is a coherent sheaf on $\mathscr{C}$ whose formation commutes with base-change by the flatness of $\cO_{\mathscr{C}}$. Then the locus of geometric points of $T$ where the morphism $f$ is representable is the open complement of the image of the support $\mathrm{Supp}(\cK) \subset \mathscr{C}$ under the proper morphism $\mathscr{C} \to T$.
    \end{itemize}
    In particular, $\mathscr{X}_{g,n}$ is algebraic and locally of finite type over $k$.
     
     Next, we show that $\mathscr{X}_{g,n}$ has affine diagonal. It is enough to show that any quasi-compact open substack $\mathscr{U} \subset \mathscr{X}_{g,n}$ has affine diagonal. By quasi-compactness, there exists some number $N$ such that for all $(\mathscr{C}, \Sigma_{\bullet}, \cL)$ families in $\mathscr{U}$, the line bundle $\cL^{\otimes N}$ descends to the coarse space $(C, \sigma_{\bullet})$ and $\cL^{\otimes N} \otimes \omega_C(\sum_{i =1}^n \sigma_i)^{\otimes N^2}$ defines an ample line bundle on $C$. Hence, we have a well-defined assignment 
    \[ F: \mathscr{U} \to \mathfrak{M}^{ss,pol}_{g,n}, \; \; \; (\mathscr{C}, \Sigma_{\bullet},\cL) \mapsto (C, \sigma_{\bullet}, \cL^{\otimes N} \otimes \omega_C(\sum_{i =1}^n \sigma_i)^{\otimes N^2})\]
    to the stack $\mathfrak{M}^{ss,pol}_{g,n}$ of semistable curves equipped with an ample line bundle. Since $\mathfrak{M}^{ss,pol}_{g,n}$ has affine diagonal \cite[Prop. 2.3]{nodal_ns_moduli}, it suffices to show that $F$ has affine relative diagonal. Let $(C_{univ}, \sigma_{\bullet}, \cM) \to \mathfrak{M}^{ss,pol}_{g,n}$ denote the universal polarized curve. We denote by $\mathscr{Z} \to C_{univ}$ the $N^{th}$-root stack of the line bundle $\cM \otimes \omega_{C_{univ}/\mathfrak{M}^{ss, pol}}(\sum_{i=1}^n \sigma_i)^{\otimes (-N^2)}$. Then $\mathscr{Z} \to C_{univ} \to \mathfrak{M}^{ss,pol}_{g,n}$ is a relatively DM morphism which is locally projective. Hence we may define the stack $\cK_{g,n}(\mathscr{Z}/\mathfrak{M}^{ss,pol}_{g,n})$ of $n$-marked twisted stable maps relative to $\mathfrak{M}^{ss,pol}_{g,n}$ as in \cite[\S8.3]{abramovich_vistoli_compactifying_stable_maps}. By the definition of $\mathscr{X}_{g,n}$, there is an inclusion $\mathscr{X}_{g,n} \hookrightarrow \cK_g(\mathscr{Z}/\mathfrak{M}^{ss,pol})$ which realizes $\mathscr{X}_{g,n}$ as the locally closed substack of $T$-families of twisted stable maps $(\mathscr{C}, \Sigma_{\bullet}) \to \mathscr{Z}_T$ such that the morphism of $T$-flat families $(\mathscr{C}, \Sigma_{\bullet}) \to ((C_{univ})_T, \sigma_{\bullet}|_T)$ is a coarse space morphism compatible with the markings. In particular, since the stack of twisted stable maps $\cK_{g,n}(\mathscr{Z}/\mathfrak{M}^{ss,pol}_{g,n})$ is separated over $\mathfrak{M}^{ss,pol}_{g,n}$ \cite[Thm.1.4.1+\S8.3]{abramovich_vistoli_compactifying_stable_maps}, the same holds for $\mathscr{X}_{g,n}$.

     We are left to show that the forgetful morphism $\bigpert_{g,n} \to \mathscr{X}_{g,n}$ is quasiaffine and of finite type. Let $\mathfrak{Y} \to \mathscr{X}_{g,n}$ denote the stack parameterizing triples $(\mathscr{C}, \Sigma_{\bullet}, \cL, s: \cL \to \cO_{\mathscr{C}})$, where $(\mathscr{C}, \Sigma_{\bullet}, \cL)$ is a point in $\mathscr{X}_{g,n}$ and $s$ is a cosection. By \cite[\href{https://stacks.math.columbia.edu/tag/08K6}{Tag 08K6}]{stacks-project}, it follows that the forgetful morphism $\mathfrak{Y} \to \mathscr{X}_{g,n}$ is affine and of finite type. Furthermore, the inclusion $\bigpert_{g,n} \hookrightarrow \mathfrak{Y}$ is an open immersion because of the condition of $s$ being nonvanishing at every generic pont of the stabilization is open. Indeed, let $\varphi: \mathscr{C} \to C^{\st}$ denote the stabilization of the coarse space, and consider the closed image $Z \subset C^{\st}$ of the vanishing locus $V(s) \subset \mathscr{C}$ under the proper morphism $\varphi$. By upper semicontinuity of fiber dimension, there is a closed subset $Z_1 \subset Z$ of points where the dimension of the fiber of $Z \to T$ is at least $1$. Then the locus of geometric points $\overline{t} \to T$ where the cosection satisfies this generic vanishing condition is the open complement of the image of the proper morphism $Z_1 \to T$.
    We conclude that the composition $\bigpert_{g,n} \hookrightarrow \mathfrak{Y} \to \mathscr{X}_{g,n}$ is quasiaffine and of finite type. 
\end{proof}

\begin{defn} \label{defn: non-stacky relatively polarized twisted curves} 
We define $\mathscr{X}_{g,n}^{\rep} \subset \mathscr{X}_{g,n}$ (resp. $\bigpert_{g,n}^{\rep} \subset \bigpert_{g,n}$) to be the substack that for any $k$-scheme $T$ parameterizes $T$-points $(\mathscr{C}, \Sigma_{\bullet}, \cL)$ of $\mathscr{X}_{g,n}$ (resp. $(\mathscr{C}, \Sigma_{\bullet}, \cL, s)$ of $\bigpert_{g,n}$) such that for all $t \in T$ the twisted curve $\mathscr{C}_t$ is a scheme. 
\end{defn}

\begin{prop} \label{prop: open substacks cartesian}
    The inclusions $\mathscr{X}^{\rep}_{g,n} \hookrightarrow \mathscr{X}_{g,n}$ and $\bigpert^{\rep}_{g,n} \hookrightarrow \bigpert_{g,n}$ are open immersions, and they fit into a Cartesian diagram
    \[
\xymatrix{\bigpert^{\rep}_{g,n}\ar[r]\ar[d] & \bigpert_{g,n}\ar[d]\\ \mathscr{X}^{\rep}_{g,n}\ar[r] & \mathscr{X}_{g,n}. }
\]
\end{prop}
\begin{proof}
    Consider the forgetful morphisms $\bigpert_{g,n} \to \mathscr{X}_{g,n} \to \mathfrak{M}^{tw}_{g, n}$ to the stack of twisted marked curves. By definition, the substacks $\mathscr{X}^{\rep}_{g,n}$ and $\bigpert^{\rep}_{g,n}$ are the preimages of the substack $\mathfrak{M}^{tw, \rep}_{g, n} \subset  \mathfrak{M}^{tw}_{g, n}$ parameterizing $T$-families $(\mathscr{C}, \Sigma_{\bullet})$ such that for all $t \in T$ the twisted curve $\mathscr{C}_t$ is a scheme. The result follows because $\mathfrak{M}^{tw, \rep}_{g, n} \hookrightarrow  \mathfrak{M}^{tw}_{g, n}$ is known to be an open immersion \cite[Cor. 1.12]{olsson-twisted-curve}.
\end{proof}

\subsection{Codimension Two Filling Conditions}

The goal of this section is to prove $\Theta$-completeness and $S$-completeness of the stacks $\mathscr{X}_{g,n}$ and $\bigpert_{g,n}$. We recall both of these codimension-two filling conditions below. Throughout, we use $R$ to denote a DVR with uniformizer $\pi$ and fraction field $K$.

\begin{defn} \label{defn: S-completeness and Theta reductivity} Let $\cX$ be an algebraic stack over $k$. We consider commutative diagrams of the form
\begin{equation} \label{eq: completeness diagram}
\xymatrix{
    \cY^\circ \ar[r] \ar[d] & \cX  \\
    \cY \ar@{-->}[ur]  
}
\end{equation}
where $\cY$ is an algebraic stack with closed point $\mathfrak{o} \in \cY$, and $\cY^\circ := \cY \setminus \mathfrak{o}$. For two different pairs $(\cY, \mathfrak{o})$, we say that the associated filling condition holds if for every such diagram \eqref{eq: completeness diagram} of solid arrows there exists a unique dotted arrow making the diagram commute.

\begin{enumerate}[(1)]
    \item Define
    $\overline{\rm{ST}}_R: = \Spec(R[s,t]/(st-\pi)) / \bG_m$ where $\bG_m$ acts by giving $s$ weight $1$ and $t$ weight $-1$. Let $\mathfrak{o}$ be the closed point corresponding to the ideal $(s, t)$. We say that $\cX$ is \emph{S-complete} if the filling condition holds for $(\cY = \overline{\rm{ST}}_R, \mathfrak{o})$ for every DVR $R$.
    \item  Define $\Theta_R := \Spec(R[t]) / \bG_m$, where $\bG_m$ acts by giving $t$ weight $-1$. Let $\mathfrak{o}$ be the closed point corresponding to the ideal $(\pi, t)$. We say that $\cX$ is \emph{$\Theta$-complete} if the filling condition holds for $(\cY = \Theta_R, \mathfrak{o})$ for every DVR $R$.
\end{enumerate}
\end{defn}

\begin{remark} \label{remark: affine diagonal reductions}
    If the stack $\cX$ has affine diagonal, then the uniqueness of the filling in \eqref{eq: completeness diagram} is automatic, and to check the existence we may replace $R$ with an arbitrary extension of DVR's $R \subset R'$ \cite[Prop. 3.17, Prop. 3.41]{alper2023existence}. We use this fact freely in our proofs of $\Theta$-completeness and S-completeness.
\end{remark}

\begin{notn} \label{notn: Y over the stack}
    In the setting of \Cref{defn: S-completeness and Theta reductivity}, each of the test stacks $\cY$ arises as a quotient $\cY = Y/\bG_m$, where for $\cY = \overline{\rm{ST}}_R$, we set $Y := \Spec(R[s,t]/(st-\pi))$ and for $\cY = \Theta_R$, we set $Y := \Spec(R[t])$. Abusing notation, we also write $\mathfrak{o} \in Y$ for the closed point lying over $\mathfrak{o} \in \cY$ cut out by the corresponding ideal $(s, t)$ in the S-completeness case and $(\pi, t)$ in the $\Theta$-completeness case. We set $Y^\circ := Y \setminus \mathfrak{o}$, so that $\cY^\circ = Y^\circ/ \bG_m$.
\end{notn}

\subsubsection*{Some reduction results for filling conditions}

\begin{lemma} \label{lemma: complete forgetful}
The forgetful morphism $\bigpert_{g,n} \to \mathscr{X}_{g,n}$ is $\Theta$-complete and $S$-complete. In other words, given any DVR $R$ and letting $\cY$ be as in \Cref{notn: Y over the stack}, for any diagram of solid arrows as below
\begin{equation} \label{eqn: s-complete forgetful}
    \xymatrix{
    \cY^\circ \ar[r] \ar[d] & \bigpert_{g,n} \ar[d]  \\
    \cY \ar@{-->}[ur]  \ar[r] & \mathscr{X}_{g,n}
}
\end{equation}
there exists a unique dotted arrow making the diagram commute. The same holds for the morphism $\bigpert_{g,n}^{\rep} \to \mathscr{X}_{g,n}^{\rep}$.
\end{lemma}
\begin{proof}
    In view of the Cartesian diagram in \Cref{prop: open substacks cartesian}, it suffices to prove the result for the morphism $\bigpert_{g,n} \to \mathscr{X}_{g,n}$. The uniqueness of the dotted arrow in \eqref{eqn: s-complete forgetful} follows from the fact that $\bigpert_{g,n}$ has affine diagonal (see \Cref{remark: affine diagonal reductions}). We just need to show the existence. The morphism $\cY^{\circ} \to \bigpert_{g,n}$ corresponds to a family $(\mathscr{C}^{\circ}, \Sigma_{\bullet}^{\circ}, \cL^{\circ}, s^{\circ})$ over $\cY^{\circ}$. By assumption, the triple $(\mathscr{C}^{\circ}, \Sigma_{\bullet}^{\circ}, \cL^{\circ})$ extends to a family $(\mathscr{C}, \Sigma_{\bullet}, \cL)$ over $\cY$. Furthermore, by Hartogs's theorem, the cosection $s^{\circ}: \cL^{\circ} \to \cO_{\mathscr{C}^{\circ}}$ extends uniquely to $s: \cL \to \cO_{\mathscr{C}}$. To conclude, we need to show that the tuple $(\mathscr{C}, \Sigma_{\bullet}, \cL, s)$ satisfies condition (c) in \Cref{defn: big stack of perturbed curves}. Consider the stabilization $(C^{\st}, \sigma_{\bullet}^{\st}) \to \cY$ of the marked twisted curve $(\mathscr{C}, \Sigma_{\bullet})$. Since $\overline{\cM}_{g,n}$ is separated and has finite inertia, the morphism $(C^{\st}, \sigma_{\bullet}^{\st}) : \cY \to \overline{\cM}_{g,n}$ factors through $\cY\to \Spec(R)$ \cite[Lem. 3.6, Lem. 3.45]{alper2023existence}. In other words, there is a family of marked stable curves $(\overline{C}^{\st}, \overline{\sigma}_{\bullet}^{\st})$ over $\Spec(R)$ such that $(C^{\st}, \sigma_{\bullet}^{\st})$ is isomorphic to the base-change $(\overline{C}^{\st}_{\cY}, (\overline{\sigma}_{\bullet}^{\st})_{\cY})$. To conclude we show that the following holds.
    
    \medskip 
    
    \noindent \textbf{Claim.} There is an open subscheme $U \subset \overline{C}^{\st}$ which is fiberwise dense over $\Spec(R)$ and satisfies the following:
    \begin{enumerate}[(a)]
        \item The morphism $\mathscr{C} \to C^{\st} \cong \overline{C}^{\st}_{\cY}$ is an isomorphism over the open $U_{\cY} \subset \overline{C}^{\st}_{\cY}$.
        \item The restriction of $s^{\circ}: \cL^{\circ} \to \cO_{\mathscr{C}_{\cY^{\circ}}}$ to $U_{\cY^{\circ}} \subset \mathscr{C}_{\cY^{\circ}}$ is an isomorphism. 
    \end{enumerate}
    
    If the claim holds, then Hartogs's theorem guarantees that $s$ is an isomorphism when restricted to $U_{\cY}$, which is fiberwise dense inside $C^{\st} \cong \overline{C}^{\st}_{\cY}$, thus concluding the proof.

    We are left to show the claim. Set $V \subset \overline{C}^{\st}$ to be the complement $\overline{C}^{\st, sm} \setminus \bigcup_{i=1}^n \overline{\sigma}^{\st}_i$ of the markings inside the smooth locus $\overline{C}^{\st,sm} \subset \overline{C}^{\st}$ of $\overline{C}^{\st} \to \Spec(R)$. Note that the morphism $\mathscr{C} \to \overline{C}^{\st}_{\cY}$ is an isomorphism when restricted to $V_{\cY}$. Let $s^{\circ}_{V}$ denote the restriction of $s^{\circ}: \cL^{\circ} \to \cO_{\mathscr{C}^{\circ}}$ to $V_{\cY^{\circ}} \subset \mathscr{C}^{\circ}$. Let $W \subset V_{\cY^{\circ}}$ denote the open subset where $s^{\circ}_V$ is an isomorphism. Since $s^{\circ}$ is the cosection corresponding to the original morphism $\cY^{\circ} \to \bigpert_{g,n}$, the open subset $W \subset \overline{C}^{\st}_{\cY^{\circ}} \cong C^{\st}_{\cY^{\circ}}$ has dense $\cY^{\circ}$-fibers. Let $H$ be the reduced closed complement of $W \subset \overline{C}^{\st}_{\cY^{\circ}}$, denote by $\pi: \cY^{\circ} \to \Spec(R)$ the structure morphism, and let $Z \subset \overline{C}^{\st}$ be the schematic closure of $H$ under the base-change $\pi_{\overline{C}^{\st}}: \overline{C}^{\st}_{\cY^{\circ}} \to \overline{C}^{\st}$. We set $U$ to be the open complement of $Z \subset \overline{C}^{\st}$. By construction $U$ satisfies conditions (a) and (b) of the claim; we are only left to prove that the $\Spec(R)$-fibers of $U \subset \overline{C}^{\st}$ are dense. Equivalently, we need to show that $Z \to \Spec(R)$ has finite fibers. To do this, we distinguish the cases of $\cY = \Theta_R$ and $\cY = \overline{ST}_R$.

    If $\cY = \overline{ST}_R$, then the scheme $\cY^{\circ}$ is the union of two copies of $\Spec(R)$ glued at their generic points. In particular, the morphism $\cY^{\circ} \to \Spec(R)$ is quasifinite and universally closed. It follows that, set-theoretically, $Z \subset \overline{C}^{\st}$ is the naive image of $H$ under $\pi_{\overline{C}^{\st}}$. Since the fibers of $H \to \cY^{\circ}$ are finite, so are the fibers of $H \to \cY^{\circ} \to \Spec(R)$, and hence the same holds for the fibers of the image $Z \subset \overline{C}^{\st}$.

    Suppose now that $\cY = \Theta_R$. Let $\eta$ denote the generic point of $\Spec(R)$. Note that $\cY^{\circ}$ is obtained by gluing $\cY^{\circ}_{t \neq 0} = \Spec(R)$ and $\Theta_{\eta}$ along their open points. Since the morphism $\pi_{\overline{C}^{\st}}$ restricts to an isomorphism over the open $\cY^{\circ}_{t \neq 0}$, it follows that the schematic image of $H|_{\cY^{\circ}_{t \neq 0}}$ has finite fibers. It is enough to show that the schematic image of $H|_{\Theta_{\eta}}$ under the morphism $ \pi_{\overline{C}^{\st}}: \overline{C}^{\st}_{\Theta_{\eta}} \to \overline{C}^{\st}_{\eta}  \hookrightarrow \overline{C}^{\st}$ is finite over $\Spec(R)$. This coincides with the closure of the (schematic) image of $H_{\Theta_{\eta}}$ under the closed projection morphism $\overline{C}^{\st}_{\Theta_{\eta}} \to \overline{C}^{\st}_{\eta}$; it suffices to show that the latter one is finite over $\eta$. For this we note that, over the open $V_{\Theta_{\eta}} \subset \overline{C}^{\st}_{\Theta_{\eta}}$, we have by assumption that the line bundle $\cL^{\circ}|_{V_{\Theta_{\eta}}}$ is isomorphic to $\cO_{V_{\Theta_{\eta}}}$ over a $\Theta_{\eta}$-fiberwise dense open. In particular, the stabilizers of the stack $V_{\Theta_{\eta}}$ act trivially on the fibers of the line bundle $\cL^{\circ}{|_{V_{\Theta_{\eta}}}}$, and therefore $\cL^{\circ}{|_{V_{\Theta_{\eta}}}}$ is isomorphic to the pullback of a line bundle $L$ on $V_{\eta}$ under the good moduli space morphism $V_{\Theta_{\eta}} \to V_{\eta}$ \cite[Thm. 10.3]{alper_good_moduli}. By the projection formula applied to $V_{\Theta_{\eta}} \to V_{\eta}$, it follows that the restriction of the cosection $s^{\circ}|_{V_{\Theta_{\eta}}}: \cL^{\circ}|_{V_{\Theta_{\eta}}} \to \cO_{V_{\Theta_{\eta}}}$ is in fact the pullback of a cosection $\overline{s}: L \to \cO_{V_{\eta}}$. Let $\overline{H} \subset V_{\eta}$ denote the finite $\eta$-scheme where $\overline{s}$ vanishes. Then $H_{\Theta_{\eta}}$ is the preimage under $\overline{C}^{\st}_{\Theta_{\eta}} \to \overline{C}^{\st}_{\eta}$ of the union $\overline{H} \cup (\overline{C}^{\st}_{\eta} \setminus V_{\eta})$ of finite $\eta$-schemes. Therefore the image of $H_{\Theta_{\eta}}$ is the finite $\eta$-scheme $\overline{H} \cup (\overline{C}^{\st}_{\eta} \setminus V_{\eta})$.
\end{proof}

\begin{lemma}\label{lemma_extending_stacky_curve}
    Let $R$ be a DVR, and let $\cY$ be either $\overline{ST}_R$ or $\Theta_R$ as in \Cref{notn: Y over the stack}. Let $\cY^{\circ} \to \mathscr{X}_{g,n}$ be a morphism corresponding to a tuple $(\mathscr{C}^{\circ}, \Sigma_{\bullet}^{\circ}, \cL^{\circ})$, and let $(C^{\circ}, \sigma_{\bullet}^{\circ})$ denote the relative coarse space of $\mathscr{C}^\circ$. Let $N>0$ be an integer such that $(\cL^\circ)^{\otimes N}$ descends to a line bundle $L^{\circ}$ on $C^{\circ}$, and consider the corresponding morphism $(C^{\circ}, \sigma_{\bullet}^{\circ}, L^{\circ}): \cY^{\circ} \to \mathscr{X}_{g,n}^{\mathrm{rep}}$. If there is an extension to a morphism $(C, \sigma_{\bullet}, L): \cY \to \mathscr{X}_{g,n}^{\mathrm{rep}}$, then there exists an extension of the original family $(\mathscr{C}^{\circ}, \Sigma_{\bullet}^{\circ}, \cL^{\circ})$ to a morphism $\cY \to \mathscr{X}_{g,n}$.
\end{lemma}
\begin{proof}
    Let us denote by $\mathscr{C}^{\circ} \to \overline{\mathscr{C}}^{\circ} \to \cY^{\circ}$ the partial relative coarse space where we only remove the fiberwise stackyness along the markings, so that $\mathscr{C}^{\circ} \to \overline{\mathscr{C}}^{\circ}$ is an isomorphism away from $\bigcup _i\Sigma^{\circ}_i$. There exists some twist $\cL^{\circ}(\sum_i a_i \Sigma_i^{\circ})$ that descends to a line bundle $\cM^{\circ}$ on $\overline{\mathscr{C}}^{\circ}$. Note that $(\cL^{\circ}(\sum_i a_i \Sigma_i^{\circ}))^{\otimes N}$ (or, equivalently, $(\cM^{\circ})^{\otimes N}$) descends to a twist $M^{\circ} := L^{\circ}(\sum_i b_i \sigma_i^{\circ})$ of $L^{\circ}$.
    
    Set $M:= L(\sum_i b_i \sigma_i)$, and let $\mathscr{R} \to C$ denote the $N^{th}$-root stack of the line bundle $M$. The morphism $\mathscr{R} \to \cY$ is proper and relatively DM with relative coarse space $C$, which is projective smooth-locally over $\cY$. Let $\cK_g(\mathscr{R}/\cY) \to \cY$ denote relative stack of twisted stable maps, which sends a scheme $T \to \cY$ to the groupoid of morphisms $f: \mathscr{S} \to \mathscr{R}_T$, where $\mathscr{S} \to T$ is a $T$-family of twisted curves of genus $g$ and $f$ is a respresentable morphism of $T$-stacks that induces a $T$-family of stable maps $\overline{f}: S \to C_T$ at the level of coarse spaces. Note that the stack $\cK_g(\mathscr{R}/\cY)$ is algebraic, relative DM and locally of finite type over $\cY$, and the assignment
    \[ \Psi: \cK_g(\mathscr{R}/\cY) \to \cK_g(C/\cY), \; \; \; [f: \mathscr{S} \to \mathscr{R}_T] \mapsto [\overline{f}: S \to C_T]\]
    is proper and quasifinite. Indeed, these properties can be checked smooth-locally over $\cY$, and when we replace $\cY$ with an atlas they follow from \cite[Thm. 1.4.1]{abramovich_vistoli_compactifying_stable_maps}.
    
    We denote by $\cK_g^{(1)}(\mathscr{R}/\cY) \subset \cK_g(\mathscr{R}/\cY)$ the open and closed substack of points $f: \mathscr{S} \to \mathscr{R}_T$ such that for all $t \in T$, the induced morphism $\overline{f}_t: S_t \to C_t$ of prestable curves of genus $g$ is an isomorphism over a dense open subset of $C_t$ \cite[Lem. 3.2]{nodal_ns_moduli}. Note that this actually forces $\overline{f}: S \to C$ to be an isomorphism. Hence the assignment $\Psi$ above induces a proper and quasifinite morphism
    \[ \cK_g^{(1)}(\mathscr{R}/\cY) \to \cY, \; \; \; [f: \mathscr{S} \to \mathscr{R}_T] \mapsto [T \to \cY].\]
    By construction the stack $\cK_g(\mathscr{R}/\cY)$ sends a scheme $T$ to the groupoid of pairs $(\mathscr{S}, \cM)$ where 
    \begin{itemize}
        \item $\mathscr{S} \to T$ is $T$-family of genus $g$ twisted curves over with coarse space $C_T$,
        \item $\cM$ is a line bundle on $\mathscr{S}$ such that the induced morphism $\mathscr{S} \to B\mathbb{G}_m$ is representable.
        \item $\cM^{\otimes N}$ descends to the line bundle $M_T$ on the coarse space $C_T$. 
    \end{itemize}
     We may view the pair $(\overline{\mathscr{C}}^{\circ}, \cM^{\circ})$ as a morphism $\cY^{\circ} \to \cK_g^{(1)}(\mathscr{R}/\cY)$ fitting into the following commutative diagram of solid arrows
    \begin{equation} \label{eqn: relative stable maps}
    \xymatrix{
    \cY^\circ \ar[r] \ar[d] & \cK^{(1)}_g(\mathscr{R}/\cY) \ar[d]  \\
    \cY \ar@{-->}[ur]  \ar[r] & \cY.
}
\end{equation}
Since the morphism $\cK^{(1)}_g(\mathscr{R}/\cY) \to \cY$ is proper and relatively DM, it follows that there exists a dotted arrow as in \eqref{eqn: relative stable maps} making the diagram commute \cite[Prop. 3.23, Prop. 3.46]{alper2023existence}. Therefore, we get an extension $(\overline{\mathscr{C}}, \cM)$ of the pair $(\overline{\mathscr{C}}^{\circ}, \cM^{\circ})$. By construction, the relative coarse space morphism $\overline{\mathscr{C}} \to C \to \cY$ is an isomorphism over the markings $\sigma_{\bullet}$. By taking appropriate root stacks of $\overline{\mathscr{C}}$ along the markings $\sigma_{\bullet}$, we obtain a family $(\mathscr{C}, \Sigma_{\bullet})$ of marked twisted semistable curves over $\cY$ extending $(\mathscr{C}^{\circ}, \Sigma_{\bullet})$. Setting $\cL := \cM|_{\mathscr{C}}(-\sum_i a_i \Sigma_i)$, we obtain the desired family $(\mathscr{C}, \Sigma_{\bullet}, \cL): \cY \to \mathscr{X}_{g,n}$ extending $(\mathscr{C}^{\circ}, \Sigma_{\bullet}^{\circ}, \cL^{\circ})$.
\end{proof}

\begin{lemma}\label{lem: no-Gm-action}
     Let $\cX$ be an algebraic stack with affine diagonal, and consider a morphism $\cY^\circ\to \cX$ for a test stack $\cY^{\circ} \subset \cY$ as in \Cref{defn: S-completeness and Theta reductivity} (1) or (2). Let $Y\to \cY$ (resp. $Y^\circ\to \cY^\circ$) be the quotient by $\bG_m$ as in \Cref{notn: Y over the stack}. If the morphism $Y^\circ\to \cY^\circ\to \cX$ extends to $Y\to \cX$, then $\cY^\circ\to \cX$ extends to $\cY\to \cX$. 
\end{lemma}

\begin{proof}
    Consider the presentation of $\cY$ given by 
    \[
   \bG_m\times \bG_m\times Y \triple\rightarrow \bG_m\times Y \rightrightarrows Y,
    \]
    and the corresponding presentation for $\cY^\circ$.
    By assumption, we have an object $x\in \cX(Y)$ which extends the object $x^\circ\in \cX(Y^\circ)$. From the presentation of $\cY^\circ$, there is an isomorphism $\xi^\circ$ between the images of $x^\circ$ via the two maps $\cX(Y^\circ)\rightrightarrows \cX(Y^\circ\times \bG_m)$, which we denote by $x_1^\circ$ and $x_2^\circ$. These two images can also be obtained as the image of $x$ via the following composition
    \[
    \cX(Y)\rightrightarrows \cX(Y\times \bG_m)\to \cX(Y^\circ\times \bG_m).
    \]
    Denote by $x_i$ the image of $x$ via $\cX(Y)\rightrightarrows \cX(Y\times \bG_m)$, which restricts to $x_i^\circ$. 
    From \cite[Lemma 4.3]{di2022degenerations}, the isomorphism $\xi^\circ\colon x_1^\circ\to x_2^\circ$ extends uniquely to an isomorphism $\xi\colon x_1\to x_2$, as $\cX$ has affine diagonal. The cocycle condition is again satisfied, from the uniqueness of the filling. Then there is a commutative diagram inducing $\cY\to \cX$ as follows 
    \[
    \cX(Y) \rightrightarrows\cX(\bG_m\times Y)  \triple\rightarrow \cX(\bG_m\times\bG_m\times Y).\qedhere\]
\end{proof}

\subsubsection*{Main result of the section}

\begin{thm}\label{T:s-complete-and-theta-red}
    The stacks $\mathscr{X}_{g,n}$ and $\bigpert_{g,n}$ are $\Theta$-complete and S-complete. The same holds for the open substacks $\mathscr{X}_{g,n}^{\rep}$ and $\bigpert_{g,n}^{\rep}$.
\end{thm}
\begin{proof}
    By \Cref{lemma: complete forgetful}, it suffices to show that both $\mathscr{X}_{g,n}$ and $\mathscr{X}_{g,n}^{\rep}$ are S-complete and $\Theta$-complete. In view of \Cref{lemma_extending_stacky_curve}, it is enough to show that $\mathscr{X}_{g,n}^{\rep}$ is S-complete and $\Theta$-reductive. This is the content of \Cref{T:s-complete} below.
\end{proof}

The rest of this section is dedicated to proving the $\Theta$-completeness and $S$-completeness of $\mathscr{X}_{g,n}^{\rep}$, which was used in the proof of \Cref{T:s-complete-and-theta-red}. In particular, for the rest of this section we work with non-twisted families of curves, which are fiberwise schemes.

\subsection{Families of curves with regular total space}

%\begin{lemma} \label{lemma: identification of semistable curves with regular total space}
%Let $R$ be a DVR, and let $C$ and $C'$ be two families of curves over $\Spec(R)$. Suppose that the following hold:
    %\begin{itemize}
        %\item The total spaces $C$ and %$C'$ are regular.
        %\item There is a family of nodal curves $C''\to \spec(R)$ with two birational and surjective morphisms $f\colon C\to C''$ and $f'\colon C'\to C''$, such that $\omega_{C/\spec(R)}$ and $\omega_{C'/\spec(R)}$ are nef over $C''$.
    %\end{itemize}
    %Then there is a unique isomorphism $\psi\colon C\to  C'$ such that the composition \[\psi\circ f'=f.\] The same results holds if we replace $\Spec(R)$ by $\mathbb{A}^1_K$ for some field $K$.
%
    %Moreover, if $\widetilde{C}\to \spec(R)$ is a family of nodal curves with smooth total space and with a birational and surjective morphism $\widetilde{C}\to C''$, then there is a factorization
    %\[
    %\xymatrix{ \widetilde{C}\ar[rr] %\ar[dr] & & C \ar[dl] \\ & C''. &}
    %\]
%\end{lemma}
\begin{lemma}\label{lemma: min resolution family of nodal curves}
    Let $R$ be a DVR {with fraction field $K$,} and let $T$ be either $\spec(R)$ or $\bA^1_K$. Consider a family of nodal curves $C\to T$ with smooth generic fiber. Then there is a unique family of nodal curves $C'\to T$ with a birational morphism $\varphi'\colon C'\to C$ over {$T$}, such that:
    \begin{enumerate}[(1)]
        \item $C'$ is regular, and
        \item {$\omega_{C'/T}$} is nef {relative to} $C$.
    \end{enumerate}
    Moreover, the morphism {$\varphi'$} is an isomorphism away from the singular locus of $C$, and $C'\to C$ is minimal in the following sense: given any family of nodal curves $C''\to T$ with a birational morphism $\varphi''\colon C''\to C$ over {$T$} and with $C''$ regular, there is a morphism {$b\colon C''\to C'$} such that $\varphi'' = \varphi'\circ b$. 
\end{lemma}
\begin{remark}
    The morphism $C'\to C$ of \Cref{lemma: min resolution family of nodal curves} is a \textit{minimal resolution}; see \cite[Thm. 2.8]{surface_resol_ref}.
\end{remark}

% \begin{remark}
%     In particular, two families of semistable nodal curves $C\to T$ and $C'\to T$, with smooth generic fiber, smooth total space, and isomorphic stabilization, are isomorphic.
% \end{remark}

\begin{remark}
    Let $R$ be a DVR, let $T$ be either $\spec(R)$ or $\bA^1_K$, and consider a family of marked nodal curves $(C,\sigma_\bullet)\to T$ with smooth generic fiber, as in \Cref{lemma: min resolution family of nodal curves}. Since the family of curves $C'\to C$ of \Cref{lemma: min resolution family of nodal curves} is an isomorphism on the smooth locus of $C \to T$, the sections $\sigma_{\bullet}$ lift to sections $\sigma_\bullet'$ of $C'\to T$.
\end{remark}

\begin{proof}[Proof of \Cref{lemma: min resolution family of nodal curves}]
{We provide a proof when $T= \Spec(R)$; the argument for $T = \mathbb{A}^1_K$ is similar.} Let $p\in C$ be a singular point of the total space of $C$. Since $C\to \spec(R)$ is a family of nodal curves, from \cite[\href{https://stacks.math.columbia.edu/tag/0CDD}{Tag 0CDD}]{stacks-project} there is an isomorphism
\[
{\cO_{C,p}^{sh}\cong (R[s,t]/(st-\pi^m))^{sh}}
\]
{where $m \ge 1$} and $\pi$ is the uniformizer of $R$. One can check by an explicit computation that:

\begin{enumerate}
    \item to resolve the singularity of {$\spec(R[s,t]/(st-\pi^m))$} when {$m > 1$}, one can consecutively perform a blow-up of the ideal sheaf of the singular point, with the reduced structure; and
    \item the fiber over $\pi=0$ has nodal singularities.
\end{enumerate}
The property of being regular and nodal over $\spec(R)$ are local in the \'etale  topology \cite[\href{https://stacks.math.columbia.edu/tag/0C5D}{Tag 0C5D}]{stacks-project} and blow-ups commute with flat (hence \'etale) base change \cite[\href{https://stacks.math.columbia.edu/tag/085S}{Tag 085S}]{stacks-project}, so the blow-ups of $C$ at its singular points, with the reduced structure, will lead to a family of nodal curves {$C' \to \spec(R)$} with {$C'$} regular. 

To check that $\omega_{C'/\spec(R)}$ is {fiberwise} trivial over $C$, it suffices to show that its restriction to every exceptional divisor has degree 0. This can again be checked for $\spec(R[s,t]/(st-\pi^m))$, where it is an explicit computation.

We argue now that $C'\to C$ is the minimal resolution of \cite[Thm. 2.8]{surface_resol_ref}; which will conclude the argument.
Let $C^{m}\to C$ be the minimal resolution of \cite[Thm. 2.8]{surface_resol_ref} and let $\xi\colon C'\to C^m$ be the map of \cite[Thm. 2.8]{surface_resol_ref}. From \cite[Thm. 9.2.2]{Liu_book}, the map $\xi$ is a composition of {$\ell$-many} blow-ups along closed points. Let's assume by contradiction that {$\ell>0$} and let $E$ be the last exceptional divisor extracted by these blowups. Then $E^2=-1$ and $\omega_{C'/\spec(R)}(E)|_{E}\cong \cO_\bP^1(-2)$ so \[\omega_{C'/\spec(R)}.E<0.\]
But $E$ is also contracted by $C'\to C$, which contradicts that $\omega_{C/\spec(R)}$ is nef over $C'$.
\end{proof}
\begin{lemma} \label{lemma: blowing up to fill to a regular family of semistable curves} Let $R$ be a DVR with algebraically closed residue field, and let $\cY$ be either $\overline{ST}_R$ or $\Theta_R$.
    Let $\widehat{C}$ be a family of nodal curves over $\Spec(R)$ with smooth generic fiber. Then there exists a family of nodal curves $C\to \cY$ such that:
    \begin{enumerate}[(a)]
        \item $C$ is regular and $\omega_{C/\spec(R)}$ is {fiberwise} trivial over $\widehat{C}\times_{\spec(R)}\cY$.
        \item There is a morphism $C\to \widehat{C}\times_{\Spec(R)} \cY$ which contracts only rational bridges.
    \end{enumerate}
    In particular, if $\widehat{C}$ is stable, then $C\to \cY$ is semistable.
\end{lemma}
\begin{proof}Consider the family of curves $\widetilde{C}\to \widehat{C}$ of \Cref{lemma: min resolution family of nodal curves}, and let $C'$ be the fiber product
\[
\xymatrix{C'\ar[r]\ar[d] & \widetilde{C}\ar[d]\\ \cY\ar[r] & \spec(R). }
\]
If $\cY=\Theta_R$, the morphism $\cY\to \spec(R)$ is smooth, so its pullback $C'\to \widetilde{C}$ is also smooth. Since $\widetilde{C}$ is regular, we have that $C'$ is regular.

If instead $\cY= \overline{ST}_R$, then $C'$ is not regular. To conclude the proof, we claim that there is a family of nodal curves $C\to \overline{ST}_R$ with total regular space and equipped with a proper birational morphism $C \to C'$ that is an isomorphism away from the singular points of $C'_{\mathfrak{o}}$. Let $Z \subset C'$ denote the union of the preimages of the nodes in the special fiber of $\widetilde{C}$. Note that the sought-after modification $C \to C'$ will be an isomorphism over the complement of $Z$. By \Cref{P:nisnevich_pushout}, in order to construct such modification $C \to C'$, we may replace $C'$ with the base-change $C'\times_{\widetilde{C}} \Spec(\widehat{\cO}_{\widetilde{C}, p})$ over the completion of $\widetilde{C}$ at each of the nodes $p \in \widetilde{C}$ in the special fiber. Since the residue field of $\spec(R)$ is algebraically closed, we have an isomorphism $\widehat{\cO}_{\widetilde{C}, p} \cong \widehat{R}\bseries{x,y}/(xy-\pi)$ \cite[\href{https://stacks.math.columbia.edu/tag/0CBX}{Tag 0CBX}]{stacks-project}, and hence an isomorphism 
\[C'\times_{\widetilde{C}} \Spec(\widehat{\cO}_{\widetilde{C}, p}) \cong \Spec(\widehat{R}\bseries{x,y}[s,t]/(xy-st))/\mathbb{G}_m.\]
where $\mathbb{G}_m$ acts trivially on $\widehat{R}\bseries{x,y}$ and with the usual weights on the coordinates $s$ and $t$. To conclude, we set $C:= B/\mathbb{G}_m$, where $B$ is the blowup of $\Spec(\widehat{R}\bseries{x,y}[s,t]/(xy-st))$ at the $\mathbb{G}_m$-equivariant closed subscheme cut out by the ideal $(x,s)$. The fact that $B$ is regular, $\cY$-flat and has nodal $\cY$-fibers follows from direct computation.

\end{proof}

For the following result we use the terminology in \Cref{notn: Y over the stack}.
\begin{coroll} \label{coroll: S-filling for families of semistable curves with regular total space}
    Let $R$ be a DVR with algebraically closed residue field, and let $\cY$ be either $\overline{ST}_R$ or $\Theta_R$. Let $(C^{\circ},\sigma_{\bullet}^{\circ})$ be a family of $n$-marked semistable curves over $\cY^{\circ}$, whose stabilization is a family of nodal marked curves. Suppose that the total space $C^{\circ}$ is regular. Then, there exists a family $(C,\sigma_{\bullet})$ of $n$-marked semistable curves over $\cY$ extending $(C^{\circ},\sigma_{\bullet}^{\circ})$ such that $C$ is regular. 
\end{coroll}
Note that the assumption that the stabilization is a family of nodal marked curves is not satisfied only when $C^{\circ}$ has genus 0 and $n<3$, or when $C^{\circ}$ has genus 1 and $n=0$.
\begin{proof}
    The stabilization of $(C^{\circ},\sigma_\bullet^{\circ})$ corresponds to a family $\cY^{\circ} \to \overline{\cM}_{g,n}$. Since $\overline{\cM}_{g,n}$ is separated and has finite inertia, this extends to a family of stable curves $\cY \to \overline{\cM}_{g,n}$ \cite[Prop. 3.23, Prop. 3.46]{alper2023existence}. Furthermore, the morphism $\cY \to \overline{\cM}_{g,n}$ factors through $\cY \to \Spec(R)$ by \cite[Lem. 3.6, Lem. 3.45]{alper2023existence}. In other words, there is a family $(C^{\st},\sigma_{\bullet}^{\st})$ over $\Spec(R)$ whose base-change to $\cY^{\circ}$
    recovers the stabilization of $(C^{\circ},\sigma_\bullet^{\circ})$. By \Cref{lemma: blowing up to fill to a regular family of semistable curves}, there is a family $(\widetilde{C},\widetilde{\sigma}_\bullet)$ of marked nodal curves with $\widetilde{C}\to \cY$ regular and such that the stabilization of $(\widetilde{C},\widetilde{\sigma}_\bullet)$ is isomorphic to $(C^{\st},\sigma_\bullet^{\st})\times_{\Spec(R)} \cY$.

    Denote by $u,v$ the pair of functions on $Y$ cutting out $\mathfrak{o}$, so that $v=t$ and $u= s$ or $\pi$ depending on whether $\cY = \overline{ST}_R$ or $\Theta_R$. Note that $\cY^{\circ}$ is the union of two opens
     \[\cY_u := \{v \neq 0\}/\bG_m \qquad \text{and} \qquad \cY_v := \{u \neq 0\}/\bG_m\]
     glued along the spectrum $\Spec(K)$ of the fraction field of $R$. These opens are isomorphic to either $\Spec(R)$, $\mathbb{A}^1_K$  or $\Spec(R) \times B\mathbb{G}_m$. 
    
    Consider the pulled-back families 
    \[
    (C_1,\sigma_{\bullet,1}):=(\widetilde{C},\widetilde{\sigma}_\bullet)\times_\cY Y\text{, }(C_1^\circ,\sigma_{\bullet,1}^\circ):=(\widetilde{C},\widetilde{\sigma}_\bullet)\times_\cY Y^\circ\text{ and }(C_2^\circ,\sigma_{\bullet,2}^\circ):=(C^\circ,{\sigma}^{\circ}_\bullet)\times_\cY Y^\circ
    \]
    where {$Y$ is as in \Cref{notn: Y over the stack}.} In particular, $Y^{\circ}$ is the union of two schemes which are either two copies of $\Spec(R)$, or a a copy of $\mathbb{G}_m \times \Spec(R)$ and $\bA^1_K$, along common open subsets. From the minimality in \Cref{lemma: min resolution family of nodal curves}, and since the family $\widetilde{C}$ satisfies that $\omega_{\widetilde{C}/\cY}$ is {fiberwise} trivial over {$C^{\st} \times_{\Spec(R)} \cY$}
    there is a map $C_2^\circ\longrightarrow C_1^\circ.$
    
    From \cite[Thm. 9.2.2]{Liu_book} this map is a composition of blow-ups. Since by assumption $(C_2^\circ,\sigma_{\bullet,2}^\circ)\to Y^{\circ}$ is semistable with smooth generic fiber, the divisor $\omega_{C_2^{\circ}}(\sigma_{\bullet, 2}^{\circ})$ is nef {relative to} $Y^{\circ}$. Therefore these blow-ups are as follows. 
    If $Y= \spec(R[s,t]/(st-\pi))$, they are supported at the intersection of some of the sections $\sigma_{\bullet,2}^{\circ}$ and the locus in $Y^{\circ}$ given by $u\cdot v=0$; we now perform these blow-ups iteratively.
    
    Consider $\Delta_1^\circ\subseteq C_1^\circ$ an irreducible component of the reduced closed subscheme of $C_1^{\circ}$ where the birational map $C_1^\circ\dashrightarrow C_2^\circ$ is not defined. From the previous paragraph and since $\Delta_1^\circ$ is assumed to be irreducible, it 
    is contained in the intersection of the image in $C_1^\circ$ of some of the sections $\sigma_{\bullet,2}^{\circ}$ and the preimage in $C_1^\circ$ of the locus in $Y^{\circ}$ given by either $u=0$ or $v=0$. In particular, $\Delta_1^{\circ}$ is a codimension-two regular and closed subscheme of $C_1^{\circ}$. Let $b_2\colon B_2^{\circ}\to B_1^{\circ}=C_1^{\circ}$ be the blow-up along $\Delta_1^\circ$. The proper transform in $B_2^{\circ}$ of the image of each section $\sigma^{\circ}_{\bullet,2}$ is still the image of a section of $B_2^{\circ}$, as from \cite[Cor. II.7.15]{Hartshorne_AG} the proper transform $\widetilde{\sigma_{i,2}(Y^{\circ})}\to \sigma_{i,2}(Y^{\circ})$ of the image of each section $\sigma_{i,2}(Y^{\circ})$ is the blow-up of a codimension one integral subscheme in $\sigma_{i,2}(Y^{\circ})$, i.e. a Cartier divisor. In particular, there are sections $\sigma_{\bullet,2}^{(2)}$ of $B_2^{\circ}\to Y^{\circ}$ such that $b_2\circ\sigma_{i,2}^{(2)}  = \sigma_{i,2}$ for every $i$. We proceed iteratively 
    and get a sequence of blow-ups
    \[
    B_\ell^{\circ}= C_2^{\circ}\to B_{\ell-1}^{\circ}\to\ldots\to B_1^{\circ} = C_1^{\circ}
    \]
    with $\Delta_1^{\circ},\ldots,\Delta_{\ell-1}^{\circ}$
    are the loci we are blowing up, with $\Delta_i^{\circ}\subseteq B_i^{\circ}$. We now construct extensions $B_i \to Y$ of the family $B_i^{\circ}$ by setting $B_1 = C_1$ and iteratively letting $B_{i+1}$ be the blow up of $B_i$ along the closure $\Delta_i$ of $\Delta_{i}^{\circ}$ in $B_{i}$. One can proceed by induction to argue that $\Delta_i$ is regular, as it is contained in a section and cut out by either $u=0$ or $v=0$, and so it is isomorphic to the locus in $Y$ cut out by either $u=0$ or $v=0$, which is regular. As before using \cite[Cor. II.7.15]{Hartshorne_AG}, the sections $\sigma_{\bullet,2}^{\circ}$ extend to sections {of $B_{\ell} \to Y$} which we denote by $\sigma_{\bullet,B}$. The pair $(B_\ell,\sigma_{\bullet,B})\to Y$ is a $Y$-family of marked curves extending $(C_2^\circ,\sigma_{\bullet,2}^\circ)$.

    Finally observe that these blow-ups are all blow-ups of $\bG_m$-equivariant loci, so the action of $\bG_m$ on $(B_\ell,\sigma_{\bullet,B})|_{Y^{\circ}}$ extends to an action on $(B_\ell,\sigma_{\bullet,B})$, giving the desired extension $(C,\sigma_\bullet)$ over $\cY$.
\end{proof}

\begin{lemma} \label{lemma: filling for genus 0 families of semistable curves with regular total space}
    Let $R$ be a DVR, and let $\cY$ be either $\overline{ST}_R$ or $\Theta_R$. Let $(C^{\circ}, \sigma_1^{\circ}, \sigma_2^{\circ})$ be a family of $2$-marked genus $0$ semistable curves over $\cY^{\circ}$. Suppose that the total space $C^{\circ}$ is regular. Then, there exists a $2$-marked family $(C, \sigma_1, \sigma_2)$ of semistable curves over $\cY$ extending $(C^{\circ}, \sigma_1^{\circ}, \sigma_2^{\circ})$ such that $C$ is regular.
\end{lemma}
\begin{proof}
    We use the notation of \Cref{notn: Y over the stack}. We write $u,v$ for the pair of parameters cutting out $\mathfrak{o} \in \cY$. The stack $\cY^\circ$ is the union of the two opens
    \[\cY_u := \{v \neq 0\}/\bG_m \qquad \text{and} \qquad \cY_v := \{u \neq 0\}/\bG_m\]
    which are glued along $\spec(K)$. If $\cY = \overline{ST}_R$, then $\cY_u \cong \cY_v \cong \Spec(R)$. If $\cY = \Theta_R$, then $\cY_u \cong \Spec(R)$ and $\cY_v \cong \Theta_K$. In the latter case, we work $\bG_m$-equivariantly on $\Spec(K[t])$, whose point $t=0$ lies over the closed point of $\Theta_K$. Over each of $\cY_u, \cY_v$, the special fiber $F$ of $C^{\circ}$ is a chain $R_1 \cup \ldots \cup R_m$ of rational curves, with $\sigma_1^{\circ}$ on $R_1$ and $\sigma_2^{\circ}$ on $R_m$.

    Choose one of the two branches, and suppose $m \ge 2$. The rational curve $R_m$ meets the rest of fiber $F$ in a single node of smoothing order one, so \[R_m.(F-R_m) = 1\text{ and } F.R_m = 0\text{, giving }R_m^2 = -1.\] We apply Castelnuovo's criterion iteratively, contracting $R_m, R_{m-1}, \ldots, R_2$, and $\sigma_2'^{,\circ}$ is sent to the image of the node $R_m \cap R_{m-1}$ each time. Doing the same on the other branch, together with the fact that they agree over $\Spec(K)$ where no contractions are performed, gives a $\bP^1$-bundle over $\cY^\circ$ equipped with two markings.

   This extends to a $\bP^1$-bundle with two markings over $\cY$ since the stack of projective bundles with two markings is isomorphic to $B\bG_m$, which is S-complete and $\Theta$-complete. We denote the corresponding $\bG_m$-equivariant family over $Y$ by $(C^{\st}, \sigma_1^{\st}, \sigma_2^{\st})$, and take $D^{\st}$ to be the divisor associated to the marked point $\sigma_2^{\st}$. This plays the role of the stabilization.

    We now realize the contraction $C^{\circ} \to C^{\st}|_{\cY^\circ}$ as a sequence of blowups{, similar to blow-ups of \Cref{coroll: S-filling for families of semistable curves with regular total space}.} On each branch of $\cY^\circ$, this contraction is a proper birational morphism of regular surfaces, contracting a chain of rational curves onto the point of $D^{\st}$ over the closed point of the branch. Applying \cite[\href{https://stacks.math.columbia.edu/tag/0C5Q}{Tag 0C5Q}]{stacks-project} shows that $C^{\circ} \to C^{\st}|_{\cY^\circ}$ is an iterated sequence of blow-ups at closed points lying over $D^{\st}|_{u = 0}$ and $D^{\st}|_{v = 0}$. We define $C \to C^{\st}$ to be the iterated sequence of blowups over all of $\cY$, where we take the centers to be the corresponding flat closures of the centers over $\cY^{\circ}$. Note that $C$ is a $\mathbb{G}_m$-equivariant family of prestable curves over $Y$ with regular total space, since the centers of the blowups always lie in smooth loci over $Y$. The markings $\sigma_1^{\circ}, \sigma_2^{\circ}$ extend uniquely as sections $\sigma_1, \sigma_2$ over $\cY$ and by construction, we have that the restriction of $(C, \sigma_1, \sigma_2)$ to $\cY^{\circ}$ recovers $(C^{\circ}, \sigma_1^\circ, \sigma_2^\circ)$.
\end{proof}

\subsection{Birational modifications of families of curves}
We gather here some further preliminary results needed for our proof $\Theta$-completeness and S-completeness.

\subsubsection{Contracting rational curves}

\begin{lemma}\label{Lem: semistable_contraction_rel_nef_is_rel_SA}
    Let $(C, \sigma_{\bullet}) \to T$ be a family of $n$-marked semistable curves of genus $g \geq 2$ over an algebraic stack $T$, and let $f: (C, \sigma_{\bullet}) \rightarrow (C^{\mathrm{st}}, \sigma^{\st}_{\bullet})$ denote its stabilization. Suppose $L$ is a line bundle on $C$ that is fiberwise nef with respect to $f$. Then the following hold:

    \begin{enumerate}[(1)]
        \item $L$ is relatively semi-ample for $f$.
        \item The scheme $\Proj_{C^{\mathrm{st}}} (\oplus_{m \ge 0}f_*(L^{\otimes m}))$ equipped with the images of the markings $\sigma_{\bullet}$ is a family of $n$-marked semistable curves over $T$, and its formation commutes with $T$-base change.
    \end{enumerate}
\end{lemma}
\begin{proof}
We work smooth-locally on $T$ and assume that $T$ is an affine scheme. Furthermore, by \cite[\href{https://stacks.math.columbia.edu/tag/0DQ0}{Tag 0DQ0}]{stacks-project}, after passing to an \'etale cover of $T$, we may assume that $C \to T$ admits a relative polarization, and so in particular $C$ is a scheme. By standard approximation results, we may furthermore assume that $T$ is Noetherian. 
%A note on this: we may write the affine $T=\Spec(A)$ as a cofiltered limit $\lim_{i \in I} \Spec(A_i)$ where $A_i$ run over all finitely generated $k$-subalgebras of $A$. By \cite[\href{https://stacks.math.columbia.edu/tag/01ZM}{Tag 01ZM}]{stacks-project}, \cite[\href{https://stacks.math.columbia.edu/tag/08K1}{Tag 08K1}]{stacks-project}, \cite[\href{https://stacks.math.columbia.edu/tag/084V}{Tag 0EY2}]{stacks-project}, \cite[\href{https://stacks.math.columbia.edu/tag/0D4K}{Tag 0C3L}]{stacks-project}, there is some finitely generated $k$-subalgebra $B \subset A$ and a family of lci curves $\overline{C} \to \Spec(B)$ such that we have an isomorphism of $T$-schemes $C \cong \overline{C}_T$. By \cite[\href{https://stacks.math.columbia.edu/tag/0B8W}{Tag 0B8W}]{stacks-project}, after perhaps enlarging $B \subset A$, we may assume that there is a line bundle $\overline{L}$ such that $\overline{L}_T \cong L$. Note that the property of the fibers of $\overline{C} \to \Spec(B)$ being semistable and of genus $g$ is open, and hence by \cite[\href{https://stacks.math.columbia.edu/tag/05F4}{Tag 05F4}]{stacks-project} after enlarging $B$ we may arrange that $\overline{C} \to \Spec(B)$ is a family of semistable curves. Finally, $\overline{L}$ being relatively nef for the stabilization morphism is a constructible property on $\Spec(B)$, and hence by \cite[\href{https://stacks.math.columbia.edu/tag/05F4}{Tag 05F4}]{stacks-project} again we may assume that the pair $(\overline{C}, \overline{L}) \to \Spec(B)$ satisfies the hypotheses of the lemma. In view of our assertion on compatibility with base-change, we are therefore reduced to the case when $T$ is affine and of finite type over $k$.
It suffices to show that the following hold:
\begin{enumerate}[(a)]
    \item For all $m \geq 0$, the pushforward $f_*(L^{\otimes m})$ is $T$-flat and its formation commutes with $T$-base-change.
    \item For $m \gg0$, the counit $f^*f_*(L^{\otimes m}) \to L^{\otimes m}$ is surjective.
    \item For all $t \in T$, the $t$-fiber of $\Proj_{C^{\mathrm{st}}} (\oplus_{m \ge 0}f_*(L^{\otimes m}))$ is a prestable curve.
\end{enumerate}
By relative cohomology and base-change as in \cite[Prop. 3]{relative_coh_base_change}, to check (a) it suffices to show that for all $t \in T$ the base-change $f_t: C_t \to C^{\st}_t$ satisfies $R^1(f_t)_*(L^{\otimes m}_t) =0$ for all $m \geq 0$. Once we check this, then (b) and (c) can be checked fiberwise over $T$. Therefore we may assume without loss of generality that $T= \Spec(k)$. Now to check (b), (c) and the vanishing of $R^1f_*(L^{\otimes m})$, we may work \'etale locally around a neighborhood of each node of $C^{\st}$. Then the result follows from a local computation where the fibers are chains of rational curves adjoined with affine opens. In this case the resulting semistable curve $\Proj_{C^{\mathrm{st}}} (\oplus_{m \ge 0}f_*(L^{\otimes m}))$ is obtained by contracting the rational curves in the chain where $L$ restricts trivially.\end{proof}

\subsubsection{Flipping rational curves}

\begin{example}\label{eg_athya_flop}
We recall how to construct the Atiyah flop. Consider the cone of a quadric $\spec(A)$, where $A = K[x,y,s,t]/(xy-st)$ over a field $K$. Let $B_1$ be the blow-up along $V(x,s)$, and $B_2$ the blow-up along $V(x,t)$. An explicit computation shows that the corresponding blow-up morphisms $B_i\to \spec(A)$ are small and each contracts a single curve isomorphic to $\bP^1$, and that $B_i$ are smooth. We will refer to the birational transformation $B_1\dashrightarrow B_2$ as the \textit{Atiyah flop}. We denote by $\widehat{B}_1 \to \Spec(\widehat{A}) \leftarrow \widehat{B}_2$ the base-change to the completion $\widehat{A}$ of $A$ at the origin $V(x,y,s,t)$.   

Observe also that, if we denote by $R_i$ the curve isomorphic to $\bP^1$ in $B_i$ contracted by $B_i\to \spec(A)$, then the blow-up of $B_i$ along $R_i$ is isomorphic to the blow-up of $\spec(A)$ at the origin $V(x,y,s,t)$.
\end{example}

\begin{lemma}\label{L:flip-rational-curves}
    Let $X \to Y$ be a family of semistable curves over a regular integral Noetherian scheme $Y$ of dimension $2$. Suppose that there is a closed point $\mathfrak{o} \in Y$ such that the fiber $X_{\mathfrak{o}}$ contains a rational bridge $R \cong \bP^1_{\kappa(\mathfrak{o})} \hookrightarrow X$, such that $X$ is regular in a neighborhood of $R$, and the normal bundle satisfies $\cN_{R/X} \cong \cO_{\bP^1}(-1)^{\oplus 2}$. Then there exist:
    \begin{itemize}
        \item a family $X' \to Y$ of semistable curves,
        \item a rational curve $R' \cong \mathbb{P}^1 \hookrightarrow X'_{\mathfrak{o}}$,
        \item $Y$-morphisms $\pi: X \to X^0$ and $\pi':X' \to X^0$ to a proper normal $Y$-algebraic space $X^0$, and
        \item a $\kappa(\mathfrak{o})$-point $p \in X^0_{\mathfrak{o}}$,
    \end{itemize}
    satisfying the following:
    \begin{enumerate}[(1)]
        \item We have $\pi(R)= \pi'(R') =\{p\}$, and the morphisms $\pi$ and $\pi'$ are isomorphisms over $X^0 \setminus \{p\}$.
        
        \item There is an identification $\widehat{\cO}_{X^0, p} \cong \widehat{A}$ with the completion of the quadric cone at the origin over the field $K =\kappa(\mathfrak{o})$ as in \Cref{eg_athya_flop}, and a commutative diagram where the two squares are Cartesian
        \[
\xymatrix{\widehat{B}_1\ar[r] \ar[d] & \spec(\widehat{A}) \ar[d] & \widehat{B}_2 \ar[d] \ar[l] \\
X\ar[r] & X^0 & X'. \ar[l]}
\]
\item Let $C \hookrightarrow X_{\mathfrak{o}} \subset X$ be another curve that meets $R$ in a single node of $X_{\mathfrak{o}}$, and let $C' \hookrightarrow X'$ denote the strict transform. For any invertible sheaf $L$ on $X$, if $L'$ denotes the corresponding invertible sheaf on $X'$, then   
\begin{equation}\label{E:flipping-degree-rule}
    \deg(L'|_{R'}) = - \deg(L|_{R}) \qquad \text{and} \qquad \deg(L'|_{C'}) = \deg(L|_{C}) + \deg(L|_{R}).
    \end{equation}
    \end{enumerate}
\end{lemma}

\begin{proof}We first prove (1) and (2).
The statement is local around $\mathfrak{o}$, so up to replacing $Y$ with a Zariski neighbourhood of it we can add an effective Cartier divisor $D$ on $X$, so that $\omega_{X/Y}(D)$ is nef and the only irreducible component of $X_{\mathfrak{o}}$ over which $\omega_{X/Y}(D)$ is zero is $R$. Observe that, from the explicit description of the normal bundle of $R$, the curve $R$ is rigid.
\Cref{Lem: semistable_contraction_rel_nef_is_rel_SA} gives a family of nodal curves which, since $R$ is rigid, contracts $R$ to a point and is an isomorphism away from $R$; let $X^0$ be such a family and let $p$ be the point to which $R$ is contracted.

As $X^0\to Y$ is a family of nodal curves, from \cite[\href{https://stacks.math.columbia.edu/tag/0CBX}{Tag 0CBX}]{stacks-project} we have
\[
\widehat{\cO}_{X^0,p}\cong \widehat{\cO}_{Y,\mathfrak{o}}\bseries{s,t}/(st-f)
\]
for $f\in \widehat{\cO}_{Y,\mathfrak{o}}$. We now argue that $f$ is the product of two generators of the maximal ideal in $\widehat{\cO}_{Y,\mathfrak{o}}$. For each node $n_1,n_2$ of $R$ the result \cite[\href{https://stacks.math.columbia.edu/tag/0CBX}{Tag 0CBX}]{stacks-project} gives:
\begin{equation}\label{eq_charts_node}
\widehat{\cO}_{X,n_1}\cong \widehat{\cO}_{Y,\mathfrak{o}}\bseries{s_1,t_1}/(s_1t_1-f_1)\text{ and }\widehat{\cO}_{X,n_2}\cong \widehat{\cO}_{Y,\mathfrak{o}}\bseries{s_2,t_2}/(s_2t_2-f_2).
\end{equation}
Since $X$ is regular, the elements $f_i$ are non-zero in $\mathfrak{m}_{\mathfrak{o}}/\mathfrak{m}_{\mathfrak{o}}^2$. Their vanishing locus in $\spec(\widehat{\cO}_{Y,\mathfrak{o}})$ is the image of the closure of the connected components of the singular locus of the morphism
\[
\mathfrak{X}:=X\times_Y\spec(\widehat{\cO}_{Y,\mathfrak{o}})\to \spec(\widehat{\cO}_{Y,\mathfrak{o}})=:\mathfrak{Y}
\]
containing $n_1$ and $n_2$. We now argue that $f_1$ and $f_2$ generate the maximal ideal $\mathfrak{m}_{\mathfrak{o}}$.

Assume by contradiction that they do not, and consider $g\in \mathfrak{m}_{\mathfrak{o}}$ such that $f_1,g$ generate $\mathfrak{m}_{\mathfrak{o}}/\mathfrak{m}_{\mathfrak{o}}^2$. Consider the following commutative diagram
\[
\xymatrix{0 \ar[r] & I_R\ar[r]&\cO_{X}\ar[r] & \cO_R\ar[r] & 0 \\
0\ar[r] &
\mathfrak{m}_{\mathfrak{o}}\ar[r] \ar[u]_\alpha & \cO_{Y,\mathfrak{o}} \ar[r] \ar[u] & k\ar[r]\ar[u]& 0.}
\]
The image of $f_1$ via $\alpha$ gives a section of $I_R/I_R^2$, which from \eqref{eq_charts_node} is non-zero, and gives an exact sequence 
\[
0\to \cO_R\to I_R/I_R^2 \to K\to 0
\]
where $K$ is the cokernel. An explicit computation using \eqref{eq_charts_node} and the fact that $f_1$ and $f_2$ have proportional image in $I_R/I_R^2$ shows that $K$ splits as $\cO_R\oplus k_{n_1}\oplus k_{n_2}$, where $k_{n_i}$ is the skyscraper sheaf at $n_i$, and the image of $g$ in $I_R/I_R^2$ is a global section for $\cO_R$. In particular, there is a surjection $I_R/I_R^2\cong\cO_{\bP^1}(1)\oplus \cO_{\bP^1}(1) \to \cO_{\bP^1}$, which is a contradiction, as such a surjection does not exist on $\bP^1$.

Since the curve we are contracting is rigid, the morphism $X\to X^0$ is an isomorphism away from $\mathfrak{o}$; we now argue that, up to replacing $f$ with $u\cdot f$ for a unit $u$, we have $f=f_1\cdot f_2$. Since $\mathfrak{Y}$ is $S_2$, it suffices to show the desired equality on $\mathfrak{Y}\smallsetminus\mathfrak{o}$. The equality follows since the set of points in $\mathfrak{Y}\smallsetminus\mathfrak{o}$ satisfying $\{f=0\}$ agrees with the image of the connected components of the nodal locus of \[\mathfrak{X}\smallsetminus \mathfrak{X}_\mathfrak{o}\to \mathfrak{Y}\smallsetminus \mathfrak{o}\] whose closure in $\mathfrak{X}$ contains $n_1$ or $n_2$. The latter is $\{f_1\cdot f_2=0\}$, so since $\mathfrak{Y}$ is locally factorial, we have $f=f_1^a\cdot f_2^b\cdot u$ for $u$ a unit. Since the contraction $X\to X^0$ is an isomorphism away from $R$, the scheme $X^{0}$ is regular in a neighbourhood of $p$ away from $p$, so again a local computation shows $a=b=1$. This shows that \begin{equation}\label{eq_completion_of_X_circ}
\widehat{\cO}_{X^0,p}\cong \widehat{\cO}_{Y,\mathfrak{o}}\bseries{s,t}/(st-f_1f_2) \cong \kappa(\mathfrak{o})[\![x,y,z,t]\!]/(xy-zt),\end{equation}
where for the last isomorphism we used \cite[\href{https://stacks.math.columbia.edu/tag/0C0S}{Tag 0C0S}]{stacks-project}.

Observe now that \[\
\Pic(\mathfrak{X}^0\smallsetminus \{p\}) = \mathbb{Z}
\]
where we denote by $\mathfrak{X}^0:=\spec(\widehat{\cO_{X^0,p}})$,
so if $\cL$ is a generator of such a group, since $X$ is regular, it extends to a line bundle $\cM$ on $\widehat{\mathfrak{X}}:=X\times_{X^0}\mathfrak{X}^0$. Since $\widehat{\mathfrak{X}}\to \mathfrak{X}^0$ is projective, we have that 
\[
\widehat{\mathfrak{X}} = \Proj\left(\bigoplus_{n\ge 0} H^0(\widehat{\mathfrak{X}},\cM^{n})\right)\text{ or }\widehat{\mathfrak{X}} = \Proj\left(\bigoplus_{n\le 0} H^0(\widehat{\mathfrak{X}},\cM^{n})\right).
\]
In particular, after identifying $\widehat{\cO}_{X^0,p}$ with $\kappa(\mathfrak{o})[\![x,y,z,t]\!]/(xy-zt)$ via \Cref{eq_completion_of_X_circ}, we can identify the two Proj above with $\widehat{B}_1$ and $\widehat{B}_2$, and up to replacing $B_1$ and $B_2$, we can assume $\widehat{\mathfrak{X}}\cong \widehat{B}_1$ over $\spec(\widehat{\cO}_{X^0,p})$.

Applying  \cite[\href{https://stacks.math.columbia.edu/tag/0GDM}{Tag 0GDM}]{stacks-project} to the map $B_2\to \spec(\kappa(\mathfrak{o})[x,y,s,t]/xy-st)$, and using again \Cref{eq_completion_of_X_circ}, gives a formal modification of $X^0$, which from \cite[\href{https://stacks.math.columbia.edu/tag/0GDU}{Tag 0GDU}]{stacks-project} gives the birational map $X'\to X^0$. These maps agree over $\mathfrak{X}^0$. 
In particular, $X'$ is regular in a neighbourhood of $R'$. Therefore, in a neighbourhood of $R'$, the morphism $X' \to Y$ is a morphism of regular schemes with equidimensional fibers of dimension $1$, and hence it is flat by miracle flatness. By the explicit description of the flop, the map $X'\to Y$ is nodal. This concludes the argument for (1) and (2).

Observe also that $\rm{Bl}_R(X) \cong \rm{Bl}_{R'}(X')$, where $\rm{Bl}_R(X)$ (resp. $\rm{Bl}_{R'}(X')$) is the blow-up of $X$ at $R$ (resp. $X'$ at $R'$). Indeed, it suffices to show that they are isomorphic over $\spec(\widehat{\cO}_{X^0,p})$, which follows from the previous identifications and \Cref{eg_athya_flop}.

    For the degree identities of \eqref{E:flipping-degree-rule}, let $Z \cong \rm{Bl}_R(X) \cong \rm{Bl}_{R'}(X')$ with projections $\rho$ and $\rho'$ onto $X$ and $X'$ respectively. Here we identify the exceptional divisor $E \cong R \times R'$ so that $\rho|_E$ and $\rho'|_E$ are the two projections. For closed points $y \in R'$ and $x \in R$, the ruling $R \times \{y\}$ maps isomorphically to $R$ under $\rho$, and $\{x\} \times R'$ maps isomorphically to $R'$ under $\rho'$. We know that
    \begin{equation}\label{E:blowup_degree}
            \rho^\ast(L) \cong (\rho')^\ast(L') \otimes \cO_Z(aE)
    \end{equation}    
    for some $a \in \bZ$. Taking the degree of both sides of \eqref{E:blowup_degree} along the curve $R \times \{y\} \hookrightarrow E$ gives $\deg(L|_R) = a \deg(\cO(E)|_{R \times \{y\}}) = - a$, because $\cO(E)|_E \cong \cO(-1,-1)$. On the other hand, restricting to $\{x\} \times R' \hookrightarrow E$ gives $0 = \deg(L'|_{R'}) - a$. This gives the first identity. For the second identity, we observe that because $C$ meets $R$ transversely at a node, the strict transform $\tilde{C} \hookrightarrow Z$ meets the exceptional divisor $E$ transversely at a single point as well, and maps isomorphically to both $C$ and $C'$. Therefore, $\deg(\cO(E)|_{\tilde{C}}) = 1$, and restricting \eqref{E:blowup_degree} to $\tilde{C}$ gives the second identity.
\end{proof}

\begin{notn}\label{notn: semistable-curves-over-Y} We let $Y$ denote either of the two test surfaces as in \Cref{notn: Y over the stack} with closed point $\mathfrak{o}$ cut out by a pair of parameters, which we write uniformly as $u, v$:
\[
(u,v)=(s,t) \ \text{ if } Y=\spec(R[s,t]/(st-\pi)),
\qquad
(u,v)=(\pi,t) \ \text{ if } Y=\spec(R[t]).
\]
We set $Y_u := \{u=0\}$ and 
$Y_v := \{v=0\}$, so that $Y_u \cap Y_v = \mathfrak{o}$. Let $(C, \sigma_\bullet) \to Y$  be a $\bG_m$-equivariant family of marked semistable curves such that the stabilization $(C^{\st}, \sigma^{\st}_\bullet) \to Y$ has smooth generic fiber. Since the family $(C, \sigma_\bullet) \to Y$ is $\bG_m$-equivariant, the stabilization descends to a morphism $(C^{\st}, \sigma^{\st}_\bullet) : \cY \to \overline{\cM}_{g,n}$ and since
$\overline{\cM}_{g,n}$ is separated and has finite inertia this morphism factors through
$\cY \to \Spec(R)$ by \cite[Lem. 3.6, Lem. 3.45]{alper2023existence}. Hence there is a family of
marked stable curves $(C^{\st}_R, \sigma^{\st}_{\bullet, R}) \to \Spec(R)$ whose base-change to $Y$
recovers $(C^{\st}, \sigma^{\st}_{\bullet})$.

Let $z_R$ be either a node or a marking of the special fiber of $C^{\st}_R$.
\begin{enumerate}
    \item Suppose $z_R$ is a node. From the description of the versal deformation space of nodal singularities as $xy = \alpha$ for some function $\alpha$ on the base $Y$ \cite[\href{https://stacks.math.columbia.edu/tag/0CBX}{Tag 0CBX}]{stacks-project}, the locus over which $z_R$ stays singular is $\{\alpha = \pi = 0\}$, and the preimage of $\pi = 0$ in $Y$ is contained in $Y_u \cup Y_v$. If instead $z_R$ is a marking, it extends to a section of $C^{\st} \to Y$ contained in the smooth locus.
    \item Let $\ell$ and $m$ be the numbers of irreducible components of the rational chains of $C$ mapping to $z_R$ over $Y_u$ and $Y_v$ respectively. We assume that the fiber $C_\mathfrak{o}$ has a rational chain over $z_R$ with $N$ irreducible components, where 
    \[N:= \ell +m+1 \ \text{if $z_R$ is a node,}\qquad N:= \ell +m \ \text{if $z_R$ is a marking.}\]
    We label the special points of this rational chain by $x_0, \ldots, x_N$, where $x_{i-1}$ and $x_i$ are joined by the rational curve $R_i \cong \bP^1$ for $1 \le i \le N$. If $z_R$ is a node, every $x_i$ is a node of $C_\mathfrak{o}$. If $z_R$ is a marking, $x_0, x_1, \ldots, x_{\ell + m-1}$ are nodes, and $x_{N}$ is the marking of $C_\mathfrak{o}$ lying over $z_R$.
    
\end{enumerate} 
\end{notn}

\begin{lemma}\label{L:always-is-atiyah-flop}
In the setting of \Cref{notn: semistable-curves-over-Y}, let $X \to Y$ be a family of nodal curves whose total space $X$ is regular, and let $R_i \cong \bP^1$ be a rational curve in the central fiber $X_{\mathfrak{o}}$, where $x_i$ and $x_{i-1}$ are both nodes of $X_{\mathfrak{o}}$, \'etale locally of the form $xy=f$ for some functions $f$ on $Y$. Suppose that one of $x_i, x_{i-1}$ is preserved over $Y_u$ and the other is preserved over $Y_v$ in the sense that $V(f) = Y_u$, respectively $V(f) = Y_v$. Then $\cN_{R_i/X} \cong \cO_{\bP^1}(-1)^{\oplus 2}$.
\end{lemma}

\begin{proof}
    After completing the DVR and replacing the ground field $k$ with its residue field, we may assume for simplicity of notation that $\kappa(\mathfrak{o}) = k$. Suppose $x_{i-1}$ is preserved over $Y_u$ and $x_i$ over $Y_v$. The two nodes are analytically locally of the form $\spec(k[\![x,y,u,v]\!]/(xy-u^{a}))$ and $\spec(k[\![x,y,u,v]\!]/(xy-v^{b}))$ as above. Since the family of nodal curves has regular total space $X$, then $u^{a}$ and $v^{b}$ must vanish with multiplicity one along $Y_u$ and $Y_v$ respectively, forcing $a = b = 1$. Thus we have $\widehat{\cO}_{X, x_{i-1}} \cong k[\![x, y, v]\!]$ with $u=xy$ and $\widehat{\cO}_{X, x_{i}} \cong k[\![x, y, u]\!]$ with $v=xy$.
    
    Let $\cI \subset \cO_X$ denote the ideal sheaf of $R_i$. We also use $u$ and $v$ to denote the corresponding sections of the rank two vector bundle $\cI/\cI^2 = \cI|_{R_i}$ on $R_i$. Let us analyze these sections at the completions of closed points of $R_i$. For a smooth point $z \in R_i$ that is not $x_i$ or $x_{i-1}$, we have $\widehat{\cO}_{X, z} \cong k[\![u, v, w]\!]$ for a local fiber coordinate $w$. Near the point $z$, we have $\cI = (u, v)$ and $u,v$ are a basis of $\cI/\cI^2$. At the node $x_{i-1}$, we have $\cI = (x, v)$ and $x,v$ are a basis of $\cI/\cI^2$. Symmetrically at $x_i$, we have $\cI = (u, y)$ and $u,y$ are a basis of $\cI/\cI^2$.

    Let $\cF_u$ and $\cF_v$ denote the saturations of the subsheaves of $\cI/\cI^2$ generated by $u$ and $v$ respectively. Locally, $\cF_u$ is generated by $x$ near $x_{i-1}$ and $u$ away from it. Similarly, $\cF_v$ is generated by $y$ near $x_{i}$ and $v$  away from it. By the previous paragraph, these local generators are a basis of $\cI/\cI^2$ at every point of $R_i$ and we thus have $\cI/\cI^2 \cong \cF_u \oplus \cF_v$. We also have that $\cF_u\cong \cO_{R_i}(x_{i-1}) \cong \cO_{\bP^1}(1)$, since $u$ vanishes to order one exactly at $x_{i-1}$. Symmetrically, $\cF_v\cong \cO_{R_i}(x_{i}) \cong \cO_{\bP^1}(1)$, since $v$ vanishes to order one at $x_{i}$. Thus, after dualizing, we get $\cN_{R_i/X} \cong \cO_{\bP^1}(-1)^{\oplus 2}$.
\end{proof}

Suppose we have a semistable $n$-marked family of nodal curves $(C, \sigma_{\bullet}) \to Y$ as in \Cref{notn: semistable-curves-over-Y} with regular total space, and that additionally we have a line bundle $L$ on $C$ whose restriction to $Y^\circ$ is nef relative to the stabilization. Since $C$ is regular, every node of $C_\mathfrak{o}$ is preserved over exactly one of $Y_u$ and $Y_v$. Set $d_i := \deg(L|_{R_i})$ for $i=1,\ldots, N$. By \Cref{L:always-is-atiyah-flop}, the curve $R_i$ may be flipped whenever one of $x_{i-1}, x_i$ is preserved over $Y_u$ and the other over $Y_v$, and \eqref{E:flipping-degree-rule} then gives
\begin{equation}\label{E:degree_update_rule}
    d_{i-1}' = d_{i-1}+d_i, \qquad d_i' = -d_i, \qquad d_{i+1}' = d_{i+1}+d_i.
\end{equation}
We distinguish the two cases of \Cref{notn: semistable-curves-over-Y}. Firstly, if $z_R$ is a node, the moves of \eqref{E:degree_update_rule} can be done for $1 \le i \le N$. On the other hand, if $z_R$ is a marking, \Cref{L:always-is-atiyah-flop} does not apply to $R_{N}$, and the moves of \eqref{E:degree_update_rule} can only be done for $1 \le i \le N-1$. In both cases we add one further integer $d_0$ to the degree sequence which represents the degree of $L$ on the rest of the fiber $C_{\mathfrak{o}}$. If $z_R$ is a node, we work modulo $N+1$, so that $d_{N+1} = d_0$.

\begin{lemma}\label{L: nef-algorithm}
     If one starts with a sequence of integers $d_0, d_1,\ldots,d_{N}$ as in the situation above and applies any sequence of moves \eqref{E:degree_update_rule}, where each move uses some $1 \leq i \leq N$ for which $d_i<0$, then this process must terminate in a sequence with $d'_i \geq 0$ for all $1 \leq i \leq N$.
\end{lemma}
\begin{proof}
    Let $D_j := \sum_{i=0}^j d_i$ for $j=0,\ldots,N$; we think of $D_j$ as a label attached to the special point $x_j$. Since $L$ is nef over $Y^\circ$, the labels $D_j$ at the nodes preserved over $Y_u$ are weakly increasing in $j$, and similarly for the labels at the nodes preserved over $Y_v$. Moreover, if $z_R$ is a marking, then $D_N \ge D_j$ for all $j$.
    
    The condition that $d_i<0$ is equivalent to $D_i < D_{i-1}$. For every $1 \le i \le N$, then \eqref{E:degree_update_rule} is equivalent to $D_{i-1}'=D_i$, $D_i' = D_{i-1}$, and $D_j'=D_j$ for all $j \neq i,i-1$. In other words, each move chooses a pair of adjacent elements $(D_{i-1},D_i)$ for $i \in 1,\ldots,N$ that are strictly decreasing and swaps them.
    
    A move only exchanges two nodes with one preserved over $Y_u$ and one over $Y_v$, and it never moves the marking. Thus, the sequence of $D_j$ labels at the nodes preserved over $Y_u$ is unchanged by a move, and similarly for $Y_v$ and for the label of the marking. Thus every move preserves that the $D_j$ at the nodes preserved over $Y_u$ are weakly increasing in $j$, and similarly for the labels at the nodes preserved over $Y_v$. Consequently, if $d_i < 0$ are some point, then $x_{i-1}$ and $x_i$ are preserved over different $Y_u$ and $Y_v$, since otherwise $D_{i-1} \le D_i$. Thus the move at $i$ is allowed. If $z_R$ is a marking in particular, $d_N \ge 0$ at every step.
    
    Each moves strictly decreases the number of pairs $j < k$ where $D_j > D_k$, so this process terminates after finitely many steps. At that point, no $d_i$ is negative for $1 \le i \le N$, so we have $D_0 \leq D_1 \leq \cdots \leq D_N$, which corresponds to $d_i = D_i - D_{i-1} \geq 0$ for $1 \leq i \leq N$.
\end{proof}

\subsection{\texorpdfstring{$\Theta$}{Theta}-completeness and \texorpdfstring{$S$}{S}-completeness}
We first prove $\Theta$- and $S$-completeness of the stack $\mathscr{X}^{\rep}_{g,n}$ under the assumption that the generic fiber is smooth, and then reduce the general case to this one. As in \Cref{notn: Y over the stack}, we use $Y$ to denote either the surface $\spec(R[s, t]/(st-\pi))$ or $\Spec(R[t])$, and $Y^\circ := Y \backslash \mathfrak{o}$ for a DVR $R$.

\begin{thm}\label{T:s-complete-generically-smooth}
    Fix a morphism $\cY^\circ \to \mathscr{X}^{\rep}_{g,n}$ with corresponding composition $Y^{\circ} \to \cY^\circ \to \mathscr{X}^{\rep}_{g,n}$. If the generic fiber of the corresponding family of curves $C^{\circ} \to Y^{\circ}$ is smooth, an extension $Y \to \mathscr{X}^{\rep}_{g,n}$ exists and is unique.
\end{thm}

\begin{proof}

\saranew{}
    Since $\mathscr{X}^{\rep}_{g,n}$ has affine diagonal, the filling is unique (see \Cref{remark: affine diagonal reductions}), so we only need to show existence. Furthermore, after passing to an extension of the DVR, we may assume that $R$ has algebraically closed residue field. 
    
    The morphism $\cY^\circ \to \mathscr{X}^{\rep}_{g,n}$ corresponds to a family $(C^\circ, \sigma^{\circ}_{\bullet}, L^{\circ})\to \cY^\circ$ with smooth generic fiber. We construct an extension $(C, \sigma_{\bullet}, L) \to \cY$ of the triple $(C^{\circ}, \sigma^{\circ}_{\bullet}, L^{\circ})$ so that $L^{\circ}$ is relatively ample for the stabilization morphism. Our proof proceeds in three steps.\\
    
    \textbf{Step 1.} Let $(C'^{, \circ}, \sigma'^{,\circ}_\bullet) \to (C^\circ, \sigma^{\circ}_\bullet)$ be the minimal resolution as in the proof of \Cref{lemma: min resolution family of nodal curves}, so that $C'^{, \circ}$ is regular. We denote by $L'^{, \circ}$ the pullback of the line bundle $L^{\circ}$ to $C'^{,\circ}$. Note that $L'^{, \circ}$ is nef relative to the stabilization. By \Cref{coroll: S-filling for families of semistable curves with regular total space}, there is an extension of $(C'^{, \circ}, \sigma'^{,\circ}_\bullet)$ to a family $(C', \sigma'_\bullet)$ of semistable curves over $\cY$ such that $C'$ is regular. Since the total space $C'$ is regular, we can extend $L'^{, \circ}$ to a line bundle $L'$ on $C'$. By construction, we have that the restriction of $(C', \sigma'_\bullet, L')$ to $\cY^{\circ}$ recovers $(C'^{, \circ}, \sigma'^{, \circ}_\bullet, L'^{, \circ})$. We abuse notation and denote by the same symbols the corresponding $\mathbb{G}_m$-equivariant families over the schemes $Y$ and $Y^{\circ}$, respectively.\\

    \textbf{Step 2.} Using \Cref{L:flip-rational-curves}, flip the irreducible components $R \cong \mathbb{P}^1$ of rational bridges in the central fiber $C'_{\mathfrak{o}}$ such that $\deg(L'|_{R}) <0$. After each flip we still get a family of semistable curves over $Y$. The rational curves $R$ in the central fiber with $\deg(L'|_{R}) <0$ all have one node preserved over $Y_u$, and one node preserved over $Y_v$, where we use the notation $(u, v)$ of \Cref{notn: semistable-curves-over-Y} for the pair of parameters cutting out $\mathfrak{o}\in Y$. Thus by \Cref{L:always-is-atiyah-flop} the flops are all Atiyah flops. If both nodes are preserved over $Y_u$ (or both over $Y_v$), we do not need to flip such $R$ since $\deg(L'|_{R}) \ge 0$ by the fact that $L'$ is nef on the rational bridges in the fibers of $C'^{, \circ}$ over $Y^{\circ}$. We can always perform a sequence of flips until we get a nef line bundle by \Cref{L: nef-algorithm}.
    Denote the family after flips still by $(C', \sigma'_{\bullet}, L')$.\\

    \textbf{Step 3.} Let $(C^{\st}, \sigma_{\bullet}^{\st})$ denote the stabilization of $(C', \sigma'_{\bullet})$. By \Cref{Lem: semistable_contraction_rel_nef_is_rel_SA}, $L'$ is semiample relative to $C^{\st}$ and we get a map $(C', \sigma'_{\bullet}) \to (C, \sigma_{\bullet})$ contracting rational curves $R$ with $\deg(L'|_{R})=0$, where $(C, \sigma_{\bullet})$ is a family of $n$-marked semistable curves over $Y$. Moreover, $L'$ descends to a line bundle $L$ on $C$ which is ample relative to $C^{\st}$; this follows directly from the relative $\Proj$ construction over $C^{\st}$ and the fact that $C' \to C$ is given by a relative linear series of a power of $L'$. The restriction of $(C, \sigma_{\bullet}, L)$ to $Y^\circ$ recovers $(C^\circ, \sigma_{\bullet}^{\circ}, L^\circ)$. Indeed, by the relative ampleness of $L^{\circ}$, the curve $C^{\circ}$ is the relative Proj of $\bigoplus_{\ell \geq 0} \st^{\circ}_*((L^{\circ})^{\otimes \ell})$, where $\st^{\circ}: (C^{\circ}, \sigma_{\bullet}^{\circ}) \to (C^{\st, \circ}, \sigma^{\st, \circ}_{\bullet})$ is the stabilization morphism.
\end{proof}

We now explain how to fill families where the generic fiber is of genus zero. For the following result, we define $\mathscr{X}^{\rep}_{0,2}$ similarly as in \Cref{defn: non-stacky relatively polarized twisted curves}, with the convention that for condition (2) of \Cref{defn: stack of general curves with line bundles}, we require $L$ to be ample relative to the base $T$.
\begin{prop}\label{T:s-complete-generically-smooth-genus-zero}
    Let $f: \cY^\circ \to \mathscr{X}^{\rep}_{0,2}$ be a family $(C^{\circ}, \sigma^{\circ}_1, \sigma^{\circ}_2, \cL^{\circ})$ of $2$-marked twisted curves of genus $0$ such that the generic fiber of $C^{\circ}$ is isomorphic to $\mathbb{P}^1$. Then, there is a morphism $\widetilde{f}: Y\to \mathscr{X}^{\rep}_{0,2}$ extending the composition $Y^{\circ} \to \cY^{\circ} \to \mathscr{X}^{\rep}_{0,2}$.
\end{prop}
\begin{proof} Let $C'^{, \circ} \to C^\circ$ be the minimal resolution as in the proof of \Cref{lemma: min resolution family of nodal curves}, so that $C'^{, \circ}$ is regular, and let $L'^{,\circ}$ be the pullback of $L^{\circ}$ to $C'^{,\circ}$. Let $\sigma_1'^{,\circ}, \sigma_2'^{,\circ}$ denote the flat closures in $C'^{,\circ}$ of the generic fibers of the markings $\sigma_1^{\circ}, \sigma_2^{\circ}$. We have that $(C'^{,\circ}, \sigma_1'^{,\circ}, \sigma_2'^{,\circ})$ is a $2$-marked family of genus $0$ curves over $\cY^{\circ}$.

By \Cref{lemma: filling for genus 0 families of semistable curves with regular total space}, we get a 2-marked family $(C', \sigma_1', \sigma_2', L')$ over $\cY$, such that the restriction of $(C', \sigma_1', \sigma_2')$ to $\cY^{\circ}$ recovers $(C'^{, \circ}, \sigma_1'^{,\circ}, \sigma_2'^{,\circ})$. Also by \Cref{lemma: filling for genus 0 families of semistable curves with regular total space}, $C'$ is regular, so the line bundle $L'^{,\circ}$ extends to a line bundle $L'$ on $C'$. We abuse notation and denote by the same symbols the corresponding $\mathbb{G}_m$-equivariant families over the schemes $Y$ and $Y^{\circ}$ respectively.
    
    The filling for the $2$-marked genus zero family on the level of coarse spaces can now be obtained in the same way as \Cref{T:s-complete-generically-smooth}. That is, flip rational curves $R$ in the central fiber where $\deg(L'|_{R}) < 0$ until $L'$ is nef, and then contract rational curves where $L'$ has degree zero using \Cref{Lem: semistable_contraction_rel_nef_is_rel_SA}. This gives a contraction $C' \to C$ where $(C, \sigma_1, \sigma_2) \to Y$ is a family of nodal curves, and $L'$ descends to a line bundle $L$ on $C$.
\end{proof}

To prove $\Theta$- and $S$-completeness of $\mathscr{X}^{\rep}_{g, n}$, we reduce to the case of \Cref{T:s-complete-generically-smooth} and \Cref{T:s-complete-generically-smooth-genus-zero} by taking the normalization of the family of curves and then gluing.

\begin{thm}\label{T:s-complete}
    The stack $\mathscr{X}^{\rep}_{g,n}$ is $\Theta$-complete and $S$-complete.
\end{thm}
\begin{proof}
Since $\mathscr{X}^{\rep}_{g,n}$ is open $\mathscr{X}_{g,n}$ by \Cref{prop: open substacks cartesian}, it has affine diagonal. Thus we only need to show existence of the filling and we may replace $R$ with an arbitrary extension of DVR's (see \Cref{remark: affine diagonal reductions}). We start with a family $(C^\circ,\sigma^\circ_\bullet, L^\circ)$ over $\cY^\circ$. By \Cref{lem: no-Gm-action}, it suffices to extend to a family $(C, \sigma_{\bullet}, L)$ over $Y$.

Take the normalization $\nu: \widetilde{C}^{\circ} \to C^{\circ}$, and let $\Delta \subset \widetilde{C}^{\circ}$ denote the closed reduced subscheme underlying the preimage of the nodal locus of $C^{\circ}$ (i.e. the conductor). Then $\Delta \to \cY^{\circ}$ is \'etale, proper and quasifinite, and it is contained in the smooth locus of $\widetilde{C}^{\circ} \to \cY^{\circ}$. After passing to an extension of the DVR $R$, we may assume that the generic fibers of every connected component of both $\Delta$ and $C^{\circ}$ are geometrically connected. This implies that $\Delta$ is a disjoint union $\Delta = \bigsqcup_i \xi_i^{\circ}$ of some even number $2m$ of markings. Considering in addition the markings $\sigma^{\circ}_{\bullet}$ already present on $C^{\circ}$, we view $(\widetilde{C}^{\circ}, \sigma_{\bullet}^{\circ}, \xi^{\circ}_{\bullet})$ as a family of disconnected curves with $n+2m$ markings indexed by the set $[n] \sqcup [2m]$. We denote by $\mathrm{inv}: [2m] \to [2m]$ the unique fixed-point free involution whose orbits are the  pairs of markings $\xi_i^\circ$ with the same image in the nodal locus of $C^{\circ}$. We have identifications $\xi_i^\circ \xrightarrow{\sim} \xi_{\mathrm{inv}(i)}^\circ$ for all $i \in [2m]$ and the pushout of $\widetilde{C}^{\circ}$ obtained by gluing these pairs recovers $C^{\circ}$.

Let $\widetilde{C}^{\circ} = \bigsqcup_{i=1}^{\ell} \widetilde{C}_i^{\circ}$ be the decomposition into connected components, so that each $\widetilde{C}_i^{\circ}$ is a $\mathbb{G}_m$-equivariant family of curves over $Y^{\circ}$ with smooth generic fiber. For each $i$, let $g_i$ denote the genus of the fibers of $\widetilde{C}_i^{\circ}$, and $n_i$ (resp. $m_i$) denote the number of markings $\sigma_{\bullet}^{\circ}$ (resp. $\xi_{\bullet}^{\circ}$) lying on $\widetilde{C}_i^{\circ}$. We may then think of $\left(\widetilde{C}_i^{\circ}, \sigma_{\bullet}^{\circ} \cap\widetilde{C}_i^{\circ}, \xi_{\bullet}^{\circ} \cap  \widetilde{C}_i^{\circ}, \nu^*(L^{\circ})|_{\widetilde{C}_i^\circ}\right)$ as a family $\cY^\circ \to \mathscr{X}^{\rep}_{g_i, n_i + m_i}$ with smooth generic fiber. 

We extend this to a family $(\widetilde{C}_i, \sigma_{\bullet} \cap\widetilde{C}_i, \xi_{\bullet} \cap  \widetilde{C}_i, \widetilde{L}|_{\widetilde{C}_i})$  in $\mathscr{X}^{\rep}_{g_i, n_i + m_i}$ over $Y$. Indeed, if $g_i \geq 2$, we use \Cref{T:s-complete-generically-smooth}. If $g_i =1$, then by construction the number of markings $n_i +m_i$ is at least $1$, and so we have access to the stabilization morphism. The same argument as in \Cref{T:s-complete-generically-smooth} then proves the existence of the extension. Finally, if $g_i =0$, then $n_i + m_i =2$, and so we may apply \Cref{T:s-complete-generically-smooth-genus-zero}. In summary, we have extended the $(n+2m)$-marked (disconnected) family $(\widetilde{C}^{\circ}, \sigma^{\circ}_{\bullet}, \xi^{\circ}_{\bullet}, \nu^*(L^{\circ}))$ to a family $(\widetilde{C}, \sigma_{\bullet}, \xi_{\bullet}, \widetilde{L})$ over $Y$.
    
We note that for each $i \in [2m]$, the automorphism $\xi^{\circ}_i \xrightarrow{\sim} \xi^{\circ}_{\mathrm{inv}(i)}$ extends uniquely to  $\xi_i \xrightarrow{\sim} \xi_{\mathrm{inv}(i)}$ over $Y$ by Hartogs's theorem. By the existence of pushouts in \cite[Thm. 1.8]{ferrans-pushouts}, we may glue these pairs of subschemes $\xi_{\bullet}$ to obtain a family $(C, \sigma_{\bullet})$ of $n$-marked curves extending $(C^{\circ}, \sigma_{\bullet}^{\circ})$. Furthermore, since $\nu^*(L)$ is pulled back from $C^{\circ}$, for each $i \in [2m]$ we obtain an identification $\nu^*(L)|_{\xi_i^{\circ}} \xrightarrow{\sim} \nu^*(L)|_{\xi_{\mathrm{inv}(i)}^{\circ}}$, where we are using the identifications $\xi^{\circ}_i \xrightarrow{\sim} \xi^{\circ}_{\mathrm{inv}(i)}$ to view these as line bundles on the same stack. By Hartogs' theorem, this extends uniquely to an identification $\widetilde{L}|_{\xi_i} \xrightarrow{\sim} \widetilde{L}|_{\xi_{\mathrm{inv}(i)}}$. Since the gluing \cite[Thm. 1.8]{ferrans-pushouts} constructed above is a pushout in the category of quasi-separated stacks, by considering the target stack $B\mathbb{G}_m$ we see that these identifications are exactly the data needed to glue the line bundle $\widetilde{L}$ to a line bundle $L$ on $C$ extending $L^{\circ}$. This yields the desired extension $(C, \sigma_{\bullet}, L)$ over $Y$.
\end{proof}

\section{Relatively polarized curves}

In this section we analyze the stack of relatively polarized curves $\pertgen_g$ introduced in \Cref{D:pertgen}, and we prove the main result, \Cref{T:perturbed_curves_main}.

\subsection{General analysis of relatively polarized curves}

Suppose $(\sC,\cL,s)$ is a degree $0$ anchored curve over an algebraically closed field, and $\sR \subset \sC$ is a rational bridge. We let $\widetilde{\mathscr{C}} := \overline{\left(\mathscr{C} \setminus \mathscr{R}\right)} \subset \mathscr{C}$ be the twisted curve whose coarse space is the partial normalization of $C^{\st}$ obtained by separating the node over which $\sR$ lies.

Recall from the introduction the following terminology: An \emph{orientation} of $\sR$ is a choice of labels $(p_-,p_+)$ for the two points $\widetilde{\sC} \cap \sR$. Given an orientation of $\sR$, we let $\ord_{\pm}(\sR)$ denote the order of vanishing of the section $s^\vee : \cO_{\widetilde{\sC}} \to \cL^\dual|_{\widetilde{\sC}}$ in an \'etale neighborhood of $p_{\pm}$ divided by the order of $\Aut(p_\pm)$. A \emph{relatively polarized curve} is a degree $0$ anchored curve $(\sC,\cL,s)$ such that $\ord_{\pm}(\sR) \leq 1$ and $\ord_+(\sR) + \ord_-(\sR) = \deg(\cL|_{\sR})$ for every rational bridge $\sR \subset \sC$. We denote by $\pertgen_g \subset \bigpert_g$ the substack of relatively polarized curves.

% \begin{defn}\label{L:pertgen}
%     The substack $\pertgen_g \subset \bigpert_g$ consists of families of relatively polarized curves such that for every geometric fiber $(\sC, \cL, s)$ and every rational bridge $\sR \subset \sC$, with some choice of orientation, one has
%     \[
%     \ord_{\pm}(\sR) \leq 1 \quad \text{and} \quad \ord_+(\sR) + \ord_-(\sR) = \deg(\cL|_{\sR}).
%     \]
% \end{defn}

% \dan{I have added this to the introduction}
% \begin{remark} \label{remark: degree on stacky curve}
%     \saranew{We are using the following notion of \emph{degree} of a line bundle on a stacky curve $\mathscr{C}$ over a field $k$. Let $\cN$ be a line bundle on $\mathscr{C}$. Then there exists some positive integer $N$ such that $\cN^{\otimes N}$ descends to the coarse space $C$. The degree of $\cN$ is defined to be $\frac{1}{N} \deg(\cN^{\otimes N})$, where here we let $\deg(\cN^{\otimes N})$ denote the total degree of $\cN^{\otimes N}$ on the prestable curve $C$.}
% \end{remark}
% \sout{By \emph{degree} of an invertible sheaf $\cL$ on an orbifold, we choose an $N$ such that $\cL^N$ descends to the coarse space, and define $\deg(\cL) = \deg(\cL^N) / N$.}

\begin{lemma}\label{L:perturbed_curve_basic_properties}
    For any geometric point $(\sC,\cL,s)$ in $\pertgen_g$, any vanishing point of the cosection $s$ and any orbifold point of $\sC$ must lie on a rational bridge of $\sC$.
\end{lemma}
\begin{proof} The condition $\ord_+(\sR) + \ord_-(\sR) = \deg(\cL|_{\sR})$ implies that if $s$ vanished at any point that was not on a rational bridge, then the total degree of $\cL^{\dual}$ on all of the bridges would be less than the total degree of $\cL^\dual$ on all of the stable components of $\sC$, which would contradict the fact that the total degree of $\cL$ is zero. As a consequence, $s$ trivializes $\cL$ at any node $x$ that is not on a rational bridge, which means that $\Aut(x)$ must be trivial, because it acts faithfully on the fiber $\cL_x$.
\end{proof}

\begin{defn}[Type of a rational bridge]\label{D:type}
    Given a geometric point $(\sC,\cL,s)$ of $\pertgen_g$ and an oriented rational bridge $\sR \subset \sC$, order the components of $\sR = \sS_1 \cup \cdots \cup \sS_m$ such that only adjacent $\sS_i$ meet and $p_- \in \sS_1$. We define the subset $\type(\sR) \subset [-1,1] \cap \bQ$ as the set $\{0,\deg(\cL|_{\sS_1}), \deg(\cL|_{\sS_1 \cup \sS_2}), \ldots, \deg(\cL|_{\sR})\} - \ord_-(\sR)$. 
\end{defn}

Note that $\max(\type(\sR))>0>\min(\type(\sR))$ for any bridge $\sR$.

\begin{lemma}
    Reversing the orientation of $\sR$ has the effect of replacing $\type(\sR) \mapsto - \type(\sR)$. So, for an unoriented rational bridge, $\type(\sR)$ is defined up to flipping signs.
\end{lemma}
\begin{proof}
    The condition that $\ord_-(\sR) + \ord_+(\sR) = \deg(\cL|_{\sR})$ implies that the definition above is equivalent to $\type(\sR) = \{-\deg(\cL|_{\sR}),-\deg(\cL|_{\sS_2 \cup \cdots \cup \sS_m}), \ldots, -\deg(\cL|_{\sS_m}), 0\} + \ord_+(\sR)$.
\end{proof}

\begin{lemma} \label{L:type_determines_isomorphism_class}
    Let $(\sC,\cL,s)$ be a geometric point of $\pertgen_g$. Let $\{\sR_i\}_{i=1}^{\ell}$ be collection of bridges of $\sC$, and choose an orientation $(p^{(i)}_-, p^{(i)}_+)$ for each $\sR_i$. Then, the following hold:
    \begin{enumerate}[(1)]
        \item Letting $\widetilde{\sC}^{st} = \overline{\sC \setminus \sqcup_{i} \sR_i}$ be the union of stable components, and defining the Cartier divisor in $\widetilde{\sC}^{st}$,
    \[\mathscr{D} := \sum_{i=1}^\ell \left(\ord_-(\cR_{i})p_-^{(i)} + \ord_+(\cR_{i})p_+^{(i)}\right),\]
    the cosection $s$ is the composition $\cL \twoheadrightarrow \cL|_{\widetilde{\mathscr{C}}^{\st}} \xrightarrow{\psi} \cO_{\widetilde{\sC}^{st}}(-\mathscr{D}) \hookrightarrow \cO_{\mathscr{C}}$ for some isomorphism $\psi: \cL|_{\widetilde{\mathscr{C}}^{\st}} \xrightarrow{\sim} \cO_{\widetilde{\sC}^{st}}(-\mathscr{D})$.
    \item Choose a bridge $\sR \subset \sC$ and consider the tuple $\type(\sR) = \left( \frac{a_j}{b_j}\right)_{j=0}^m$, where we express each term in reduced fraction form. Then the $j^{th}$ irreducible component $\sS_j$ is isomorphic to the stack $\mathbb{P}^1\left[\frac{1}{b_{j-1}} \mid \frac{1}{b_j}\right]$ obtained from $\mathbb{P}^1$ by taking the root $b_{j-1}^{th}$-root stack at $0$ and the $b_j^{th}$-root stack at $\infty$. Under this identification, we have $\cL|_{\sS_j} \cong \cO(\frac{a_j}{b_j} \cdot [\infty]- \frac{a_{j-1}}{b_{j-1}} \cdot [0])$. In other words, the bridge $\sR$ and line bundle $\cL|_{\sR}$ are as pictured in \Cref{fig: first picture combinatorial proof} below.

      \begin{figure}[H]
    \centering
    \caption{}
    \vspace{0.5cm}
    \begin{circuitikz}

        \node [font=\Large] at (-2,0) {$\mathscr{R}=$}; 
        
        \draw [ fill={rgb,255:red,255;
        green,235; blue,235} , line width=1pt ] (0,0) ellipse (1 and 0.5) node [font = \small]
        {} ; \node  at (0,1) {$\mathscr{S}_1$}; 

        \node at (-1,0) {$\bullet$};

        \node at (-1,-0.7) {$\mu_{b_0}$};
        
        \draw [
        fill={rgb,255:red,255; green,235; blue,235} , line width=1pt ] (2,0)
        ellipse (1 and 0.5)
        node [font = \small] {} ; \node  at (2,1) {$\mathscr{S}_2$};

        \node at (1,0) {$\bullet$};

        \node at (1,-0.7) {$\mu_{b_1}$};

        \draw[ fill={rgb,255:red,255; green,235; blue,235} , line width=1pt ]
        (4,0.5) arc(90:270: 1 and 0.5); \node  at (4,1) {$\mathscr{S}_3$};

        \node at (3,0) {$\bullet$};

        \node at (3,-0.7) {$\mu_{b_2}$};

        \node  at (5,0) {$\dots$};

        \draw[ fill={rgb,255:red,255; green,235; blue,235} , line width=1pt ]
        (6,-0.5) arc(-90:90: 1 and 0.5); \node  at (6,1) {$\mathscr{S}_{m-1}$};

        \draw [ fill={rgb,255:red,255; green,235; blue,235} , line width=1pt ]
        (8,0) ellipse (1 and 0.5) node [font = \small] {} ; \node  at (8,1)
        {$\mathscr{S}_m$};

        \node at (7,0) {$\bullet$};

        \node at (7,-0.7) {$\mu_{b_{m-1}}$};

        \node at (9,0) {$\bullet$};

        \node at (9,-0.7) {$\mu_{b_{m}}$};

        \node [font=\Large] at (-2,-2) {$\cL|_{\mathscr{R}} =$}; 

        \node  at (0.8,-2) {$\cO( \frac{a_1}{b_1} \cdot [\infty] - \frac{a_0}{b_0}\cdot[0])$};

        \node  at (5,-2) {$\cdots$};

        \node  at (8,-2) {$\cO(\frac{a_m}{b_m} \cdot [\infty]- \frac{a_{m-1}}{b_{m-1}}\cdot[0])$};

        \node at (7,0) {$\bullet$};

        \node at (7,-0.7) {$\mu_{c_{m-1}}$};
    \end{circuitikz}  
    \label{fig: first picture combinatorial proof}
\end{figure}
    \end{enumerate}
\end{lemma}
\begin{proof}
    Part (1) is a direct consequence of \Cref{L:perturbed_curve_basic_properties} and the definition of $\ord_-(\sR_i)$ and $\ord_+(\sR_i)$ as orders of vanishing of the cosection at the endpoints of the bridge $\sR_i$, using that the relative ampleness of $\cL|_{\sR_i}$ forces the cosection to vanish on every rational bridge $\sR_i$.

    For Part (2), we note that we must have $\mathscr{S}_j \cong \mathbb{P}^1\left[ \frac{1}{d_{j-1}} \mid \frac{1}{d_j}\right]$ for some tuple of positive integers $(d_j)_{j=0}^m$. We use freely the description of the Picard group of $\mathbb{P}^1\left[ \frac{1}{d_{j-1}} \mid \frac{1}{d_j}\right]$ given in \cite[Prop. 2.2]{martens2015variationsthemegrothendieck}, and the fact that it is generated by the line bundles $\cO(\frac{1}{d_{j-1}}\cdot [0])$ and $\cO(\frac{1}{d_j}\cdot[\infty])$. By the definition of $\ord_-(\sR_i)$ and the fact that the isotropy group of $p^i_-$ acts faithfully on the fiber of $\cL$, it follows the isotropy group of $p^i_- = 0 \in \mathscr{S}_1$ is $\mu_{b_0}$ and the corresponding character of the fiber of $\cL$ is the same as for $\cO(-\frac{a_0}{b_0}\cdot [0])$. Since the degree of $\cL|_{\mathscr{S}_1}$ is $\frac{a_1}{b_1}- \frac{a_0}{b_0}$ and the isotropy group of $\infty \in \mathscr{S}_1$ acts faithfully on the fiber of $\cL$, this forces the isotropy group of $\infty$ to be $\mu_{b_1}$ and $\cL|_{\mathscr{S}_1}\cong \cO(\frac{a_1}{b_1} \cdot [\infty] - \frac{a_0}{b_0} \cdot [0])$. It follows that the isotropy group of $0 \in \mathscr{S}_2$ is $\mu_{b_1}$ and the corresponding character of the fiber of $\cL$ agrees with the one for $\cO(-\frac{a_1}{b_1} \cdot [0])$. Using that the degree of $\cL|_{\mathscr{S}_2}$ is $\frac{a_2}{b_2} - \frac{a_1}{b_1}$, we conclude that $\mathscr{S}_2 \cong \mathbb{P}^1\left[ \frac{1}{b_{1}} \mid \frac{1}{b_2}\right]$ and $\cL|_{\mathscr{S}_2}  \cong \cO(\frac{a_2}{b_2} \cdot [\infty] - \frac{a_1}{b_1} \cdot [0])$. Iterating this procedure yields the desired result for all components $\mathscr{S}_j$.
\end{proof}

\subsubsection*{Rational models for nodes}

Let $P^{\rm{un}} \to \bA^1$ denote the versal deformation of the nodal curve $\bP^1 \cup \bP^1$, i.e., it is an isotrivial degeneration of one line to two lines meeting at a node $x \in P^{\rm{un}}$. The family $P^{\rm{un}} \to \bA^1$ can be realized as the composition of the blow-up $B\to \bP^1\times \bA^1$ at a point, with the second projection $\bP^1\times \bA^1\to \bA^1$. We label the components of the zero fiber as $P^{\rm{un}}_0 = P^{\rm{un},-}_{0} \cup P^{\rm{un},+}_0$, which equips any rational bridge in a twisted curve $\sP \to P_0^{un}$ with an orientation.

The versality of the family $P^{\rm{un}} \to \bA^1$ means the following: For any family of nodal curves $C \to \Spec(R)$ for a complete local Noetherian algebra $R$ with residue field $k$, if $x \in C(k)$ is a node lying over the closed point, then there is a morphism $\Spec(R) \to \bA^1$ taking the closed point to $0$ and an isomorphism $\varphi : \widehat{C} \cong \widehat{P_R}$ between the formal completion of $C$ at $x$ and the formal completion of $P_R := \Spec(R) \times_{\bA^1} P^{\rm{un}}$ at the unique node over the closed point of $\Spec(R)$ \cite[\href{https://stacks.math.columbia.edu/tag/0CBX}{Tag 0CBX}]{stacks-project}. We call this data a \emph{rational model} around the node $x \in C(k)$. The preimage of the node in $P^{\rm{un}}_0$ is a closed subscheme $Z \hookrightarrow P_R$ such that composing with the projection to $\Spec(R)$ is also a closed immersion $Z \hookrightarrow \Spec(R)$, hence $\cO_Z$ is $\pi$-adically complete, and therefore $\varphi$ identifies $Z$ with a closed subscheme of $C$ as well.

The following lemma applies to a complete local Noetherian algebra $R$ with residue field $k$, a family of twisted curves $\sC \to Y$ over a Noetherian $R$-scheme $Y$, an invertible sheaf $\cL$ on $\sC$, and a cosection $s : \cL \to \cO_\sC$. We suppose that there is a stable family of nodal curves $C^{\st} \to \Spec(R)$ and a morphism $\pi : \sC \to C^{\st}$ such that the morphism $\sC \to Y \times_{\Spec(R)} C^{\st}$ is the stabilization in every fiber over $Y$.

\begin{lemma}\label{L:rational_model_type_calculation}
    Let $x \in C^{\st}(k)$ be a node, and let $\Spec(R) \to \bA^1$ and $\varphi : \widehat{C^{\st}} \cong \widehat{P_R}$ be a rational model around $x$. Assume that $V(s) \setminus \pi^{-1}(Z) \subset |\sC|$ is closed. Then the pushout
    \[\xymatrix{\sC \times_{C^{\st}} (\Spec(\widehat{\cO}_{C^{\st},x}) \setminus Z)  \ar[r]^-{\varphi} \ar[d] & Y \times_{\Spec(R)} (P_R \setminus Z) \ar[d] \\ \sC \times_{C^{\st}} \Spec(\widehat{\cO}_{C^{\st},x}) \ar[r] & \sP}\]
    exists in the category of algebraic stacks with affine diagonal, and $\sP$ is a family of twisted chains of rational curves over $Y$. Furthermore:
    \begin{itemize}
    \item $\cL$ and $s$ glue uniquely with $(s = 1 : \cO_{P_R \setminus Z} \cong \cO_{P_R \setminus Z})$ to obtain a line bundle $\cL_\sP$ and cosection $s_\sP : \cL_\sP \to \cO_\sP$ on $\sP$.
    \item For every $y \in Y$, the choice of orientation for $P_0^{\rm{un}}$ induces an orientation for the rational bridge $\sR_y := \sC_y \cap \pi^{-1}(Z) \subset \sC_y$ and picks out a distinguished end component of the fiber $\sP_y^- \subset \sP_y$, and one has
    \begin{gather*}
        %\type(\sR_y) = \{\deg(\cL|_{\sP_y^{-}}),\deg(\cL|_{\sP_y^{-} \cup \sS_1}), \deg(\cL|_{\sP_y^{-} \cup \sS_1 \cup \sS_2}), \ldots, \deg(\cL|_{\sP_y^{-} \cup \sR_y})\}.
        \ord_\pm(\sR_y) = -\deg(\cL_{\sP}|_{\sP_y^\pm}) \text{ and}\\
        \type(\sR_y) = \left\{\deg(\cL_\sP|_{\sS}) \mid \text{connected proper subcurves } \sS \subsetneq \sP_y \text{ that contain } \sP_y^- \right\}.
    \end{gather*}
    \end{itemize}
\end{lemma}
\begin{proof}
    The existence of the pushout is \Cref{P:nisnevich_pushout}, and it is a twisted curve because the pushout description itself can be used to give an \'etale local identification of $\sP$ with families of twisted curves over $Y$. The existence and uniqueness of the gluing $s_\sP : \cL_\sP \to \cO_{\sP}$ follows from the universal property of pushouts, applied to maps to $\bA^1 / \bG_m$ -- the key fact is that $s$ is nowhere vanishing when restricted to the upper-left corner of the pushout diagram. We have $\ord_-(\sR_y) = -\deg(\cL_\sP|_{\sP_y^-})$, because if a section of a line bundle on an orbifold curve vanishes at a single point, then the orbifold notion of order of vanishing agrees with the orbifold notion of degree for the line bundle. The final expression for $\type(\sR_y)$ results from the fact that $\sP_y \cong \sP_y^- \cup \sR_y \cup \sP_y^+$ as polarized chains.
\end{proof}

The new interpretation of $\type(\sR_y)$ in terms of degrees of $\sL$ on the rational model is the key idea leading to all of the results below.

\subsection{Openness}

\begin{lemma}\label{L:pertgen_closed_under_generization}
Let $(\sC,\cL,s)$ be an $R$-point of $\bigpert_g$, where $R$ is a complete DVR. If the special fiber lies in $\pertgen_g$, then the whole family lies in $\pertgen_g$.    
\end{lemma}
\begin{proof}
We denote by $\eta$ and $0$ the generic and special points of $\Spec(R)$, respectively. Let $\sR_\eta \subset \sC_\eta$ be a rational bridge, and let $\sS := \overline{\sR_\eta} \cap \sC_0 \subset \sC_0$ be its flat limit. Let $\sR_0 \subset \sC_0$ be the rational bridge that contains $\sS$, oriented compatibly with $\sR_\eta$. It decomposes as $\sR_0 = \sR_0^- \cup \sS \cup \sR_0^+$. Let $D_\pm \subset V(s_\eta) \setminus \sR_\eta \subset \sC_\eta$ be the part of the vanishing locus whose limit in $\sC_0$ lies in $\sR^\pm_0\setminus \sS$. This part of the vanishing locus must be a collection of isolated points lying over the smooth locus of the stabilization of $\sC$, so we may regard $D_-$ and $D_+$ as effective divisors on $\sC_\eta$. The lemma follows from the following key formula
\begin{equation}\label{E:order_inequality}
    \ord_\pm(\sR_\eta) = \ord_\pm(\sR_0) - \deg(\cL|_{\sR_0^{\pm}}) - \deg(D_\pm).
\end{equation}

To see why this implies the lemma, the hypothesis that the special fiber lies in $\pertgen_g$ implies that $\ord_\pm(\sR_0) \leq 1$. Because $\ord_\pm(\sR_\eta) \geq 0$ and $\deg(D_\pm) \in \bZ_{\geq 0}$, this implies that $\deg(D_\pm) = 0$, i.e., $D_\pm$ is empty -- this is the main point where the restriction $\ord_\pm(\sR) \leq 1$ in \Cref{D:pertgen} is needed. \eqref{E:order_inequality} also implies that $\ord_\pm(\sR_\eta) \leq \ord_\pm(\sR_0) \leq 1$. Finally, adding both versions of \eqref{E:order_inequality}, and using the equality $\deg(\cL|_{\sR_\eta}) = \deg(\cL|_{\sS})$, we have
\[
\ord_-(\sR_\eta) + \ord_+(\sR_\eta) - \deg(\cL|_{\sR_\eta}) = \ord_-(\sR_0) + \ord_+(\sR_0) - \deg(\cL|_{\sR_0^- \cup \sS \cup \sR_0^+}) = 0.
\]
Applying these two conclusions to every rational bridge in $\sC_\eta$ shows that the generic fiber lies in $\pertgen_g \subset \bigpert_g$.

To prove \eqref{E:order_inequality}, we can choose another family of twisted curves $\sC' \to \Spec(R)$ and a map $\sC' \to \sC$ that is an isomorphism over the generic fiber and contracts rational subchains in the special fiber, and such that the closure of $D_\pm$ meets the smooth locus of $\sC'_0$. %\dan{I am pretty sure this last thing is true, but it deserves more explanation.}
It suffices to prove \eqref{E:order_inequality} for the pullback of $s : \cL \to \cO_{\sC}$ to $\sC'$, so we replace $\sC$ with $\sC'$ and assume the closure of $D_\pm$ lies in the smooth locus of $\sR_0^\pm$.

The fact that $s$ vanishes along $D$ to the correct orders implies that we can factor $s = s' \cdot f$ for the canonical $f : \cO_{\sC}(-\overline{D}_\pm) \to \cO_{\sC}$. By construction, $s' : \cL(\overline{D}_\pm) \to \cO_{\sC}$ is a cosection that does not vanish at any point in the generic fiber whose limit lies in $\sR_0$. We may therefore apply \Cref{L:rational_model_type_calculation} to the family $(\sC,\cL(\overline{D}_\pm),s')$ over $Y = \Spec(R)$, to produce an oriented family of rational chains $\sP \to \Spec(R)$ that models the degeneration of bridges. Also, because the cosection $f$ is nonvanishing at the endpoints of $\sR_0$ and $\sR_\eta$, the value of $\ord_{\pm}$ is the same for this new family of curves. The closure of the end components $\sP_\eta^\pm$ of the generic fiber is a flat degeneration of $\sP_\eta^\pm$ to $\sR_0^\pm \cup \sP_0^\pm$. The equation \eqref{E:order_inequality} then follows from the fact that $\deg(\cL(\overline{D_\pm})|_{\sP^\pm_\eta}) = -\ord_{\pm}(\sR_\eta)$ and $\deg(\cL(\overline{D_\pm})|_{\sP_0^{\pm}}) = -\ord_\pm(\sR_0)$.
\end{proof}

\begin{lemma}\label{L:types_for_degenerations}
Suppose $(\sC,\cL,s)$ is an $R$ point of $\pertgen_g$ for a complete DVR $R$. If $\sR_\eta \subset \sC_\eta$ is a rational bridge in the generic fiber, and $\sR_0 \subset \sC_0$ is the unique rational bridge in the special fiber that contains $\overline{R_\eta} \cap \sC_0$, then for the orientation of $\sR_0$ induced by an orientation of $\sR_\eta$,
\[\type(R_\eta) \subset \type(R_0).\]
\end{lemma}
\begin{proof}
    After replacing $R$ with a finite extension we may assume that the stabilization of the family $C^{\st} \to \Spec(R)$ has a section $x \in C^{\st}(R)$ whose generic fiber $x_\eta$ is the node over which $R_\eta$ lies. Then $\sR_0$ is the preimage of the node $x_0$ in $C_0^{\st}$ and $\overline{\sR_\eta} \cap \sC_0 \subset \sR_0$ is a subchain of rational curves.
    
    %The endpoints $(p_\pm)_\eta$ of $\sR_\eta$ might specialize to internal nodes $x_\pm \in \sR_0$, but there is a unique way to label the endpoints $(p_\pm)_0$ of $\sR_0$ such that $(p_-)_0$ is closer to $x_-$ than $x_+$ and $(p_+)_0$ is closer to $x_+$ than $x_-$.

    We choose a rational model for the node $x_0 \in C^{\st}$ over $\Spec(R)$ compatible with the chosen orientation of $\sR_\eta$. \Cref{L:rational_model_type_calculation} applied with $Y=\Spec(R)$ gives a family of chains of rational orbifold curves $\sP \to \Spec(R)$ such that for either point $z \in \Spec(R)$ the set $\type(\sR_z) \subset \bQ$ is exactly the set of degrees of connected proper subchains of $\sP_z$ that contain $\sP^-_z$. The flat limit of any subchain of this form in $\sP_\eta$ is a subchain of the same form in $\sP_0$, hence $\type(\sR_\eta) \subset \type(\sR_0)$.
\end{proof}

\begin{prop}\label{prop: perturbed-curves-open}
    For all $\cB$ as in \Cref{H:stability_condition}, the substack $\pertgen_g^{\cB} \subset \pertgen_g$ is open. In particular,  $\pertgen_g \subset \bigpert_g$ is open.
\end{prop}
\begin{proof}
    \Cref{L:pertgen_closed_under_generization} and \Cref{L:types_for_degenerations} establish that $\pertgen^\cB_g \subset \bigpert_g$ is closed under generization, so to show that it is open, it suffices to show that $|\pertgen^\cB_g| \subset |\bigpert_g|$ is a locally constructible subset. We use the natural gluing morphisms
    \begin{gather*}
        \bigpert_{g_1,1} \times_{\bA^1/\bG_m} \mathscr{X}_{0,2} \times_{\bA^1/\bG_m} \bigpert_{g_2,1} \to \bigpert_{g_1+g_2} \text{ and}\\
        \bigpert_{g,2} \times_{(\bA^1/\bG_m)^2} \mathscr{X}_{0,2} \to \bigpert_{g+1},
    \end{gather*}
    where $\mathscr{X}_{0,2}$ denotes the stack of marked rational (twisted) chains, regarded as a stack over $(\bA^1/\bG_m)^2$ by the morphism $\mathscr{X}_{0,2} \to (B\bG_m)^2$ taking a polarized marked chain to the line bundle at its endpoints composed with the inclusion $(B\bG_m)^2 \hookrightarrow (\bA^1/\bG_m)^2$.
    
    The complement $|\pertgen^\cB_g|^c$ is the union of the images under these gluing morphisms of two kinds of subsets of the product: 1) the locus of points where one marked point lies on a stable component and $\ord(s) > 1$ there, and 2) the locus of points where both marked points lie on a stable component, and $\ord_{p_1}(s_1) + \ord_{p_2}(s_2) \neq \deg(\cL)$. The condition that the marked point lies on a stable component is open, and the function $\ord_p(s)$ is constructible there, so the first kind of subset is locally constructible. Likewise, the second kind of subset is locally constructible because $\deg(\cL)$ is a constructible function on $\mathscr{X}_{0,2}$. The local constructibility of $|\pertgen^\cB_g|^c$, and therefore the local constructibility of $|\pertgen^\cB_g|$, then follows from the fact that the gluing morphisms are of finite type.
\end{proof}

\subsection{Deformation theory of relatively polarized curves}

\begin{obs} \label{obsevation: deformation theory of pert curves} Consider the forgetful morphism 
\[F: \pertgen_g \to \mathfrak{M}^{tw,ss}_g, \; \; (\mathscr{C}, \cL, s: \cL \to \cO_{\mathscr{C}}) \mapsto \mathscr{C}.\]
By the proofs of \Cref{prop: big stack of perturbed curves is algebraic} and \Cref{prop: perturbed-curves-open}, $\pertgen_g$ is an open substack of the mapping stack $\Map_{\mathfrak{M}^{tw,ss}_g}(\mathscr{C}_{univ}, \mathbb{A}^1/\mathbb{G}_m)$, where $\mathscr{C}_{univ} \to \mathfrak{M}^{tw,ss}_g$ denotes the universal family of twisted semistable curves and $\mathbb{G}_m$ acts linearly on $\mathbb{A}^1$ with weight $-1$. This mapping stack has a natural derived enhancement, given by the evident extension of the moduli functor to derived schemes as in \cite[\S 5.1]{halpernleistner2019mapping}. Therefore, $\pertgen_g$ also admits a natural derived enhancement, which we denote by $\pertgen_g$ by a small abuse of notation. The description of the (co)tangent complex $\mathbb{T}_{\pertgen_g/\mathfrak{M}^{tw,ss}_g}$ in \cite[Prop. 5.1.10]{halpernleistner2019mapping} implies that for all field-valued points $p = (\mathscr{C}, \cL, s: \cL \to \cO_{\mathscr{C}})$ we have
\[ p^*\mathbb{T}_{\pertgen_g/\mathfrak{M}^{tw,ss}_g} \cong R\Gamma\left(\mathscr{C}, [\cO_{\mathscr{C}} \xrightarrow{s^{\vee}} \cL^{\vee}]\right),\]
where we are taking the hypercohomology of a two term complex of sheaves on $\mathscr{C}$ with $\cO_{\mathscr{C}}$ in cohomological degree $-1$. It follows in particular that the derived stack $\pertgen_g$ is quasi-smooth over the smooth stack $\mathfrak{M}^{tw,ss}_g$. For a line bundle $\cL$ on an unmarked twisted curve $\sC$ of genus $g$, one has $\chi(L^{\otimes n}) = n \deg(L) + 1-g$ for all integers $n$ \cite[Thm.~7.2.1]{AGV}, which implies that $\chi(p^*\mathbb{T}_{\pertgen_g/\mathfrak{M}^{tw,ss}_g}) = 0$ for all points as above.
\end{obs}

\begin{prop} \label{prop: pert curves stack is integral}
     The following hold:
    \begin{enumerate}[(1)]
        \item The stabilization morphism $\st: \pertgen_g \to \overline{\cM}_g$ is syntomic. In particular, $\st$ is flat and $\pertgen_g$ is lci.
        \item The stack $\pertgen_g$ is geometrically integral.
    \end{enumerate}
    In particular, for all $\B$ as in \Cref{H:stability_condition}, the open substack $\pertgen_g^{\cB} \subset \pertgen_g$ also satisfies (1) and (2) above.
\end{prop}
\begin{proof} After extending the base field, we may assume that $k$ is algebraically closed.

\medskip

    \noindent (1). By \Cref{obsevation: deformation theory of pert curves}, $\pertgen_g$ admits a quasi-smooth derived enhancement relative to the smooth stack $\mathfrak{M}^{tw,ss}_g$, and hence the derived enhancement is (absolutely) quasismooth. Furthermore, the virtual rank of $\mathbb{T}_{\pertgen_g/\mathfrak{M}^{tw,ss}_g}$ is $0$ at every point. By the distinguished triangle  
    \[ \mathbb{T}_{\pertgen_g/\mathfrak{M}^{tw,ss}_g} \to \mathbb{T}_{\pertgen_g} \to F^*\mathbb{T}_{\mathfrak{M}^{tw,ss}_g},\]
    the virtual rank of $\mathbb{T}_{\pertgen_g}$ at any point agrees with the virtual rank of $\mathbb{T}_{\mathfrak{M}^{tw,ss}_g}$, which is $3g-3$. Note that this agrees with the dimension of the smooth integral stack $\overline{\cM}_g$.
    
    To show that the morphism $\st: \pertgen_g \to \overline{\cM}_g$ is syntomic, it suffices to show that the fibers of $\st: \pertgen_g \to \overline{\cM}_g$ have dimension at most $0$. We have a factorization $\st: \pertgen_g \xrightarrow{f} \mathfrak{M}^{tw,ss}_g \xrightarrow{\st} \overline{\cM}_g$. The morphism $\st: \mathfrak{M}^{tw,ss}_g \to \overline{\cM}_g$ is syntomic of relative dimension $0$, and for all $p \in \overline{\cM}_g(k)$ the locus $\mathfrak{M}^{tw,ss,=\ell}_g \cap \st^{-1}(p) \hookrightarrow \st^{-1}(p)$ where we have $\ell$ stacky rational bridges is locally closed of codimension $\ell$ in $\st^{-1}(p)$. It therefore suffices to show that the preimage $f^{-1}(x)$ of any point $x \in \cM^{tw,ss,=\ell}_g(k)$ has dimension at most $\ell$.

    Fix such a point $x = \mathscr{C}$, and let $\mathscr{R}_1, \ldots \mathscr{R}_\ell$ denote its rational bridges. Choose an orientation $(p^i_-, p^i_+)$ for every bridge $\sR_i$. Let $\widetilde{\mathscr{C}}^{\st} \subset \mathscr{C}$ denote the closure of the complement of the rational bridges as in \Cref{L:type_determines_isomorphism_class}.
     Each $k$-point $(\mathscr{C}, \cL, s)$ in $f^{-1}(x)$ determines a tuple of types $\sigma := (\type(\cR_i))_{i=1}^{\ell}$, where each $\type(\cR_i)$ is a subset of $[-1, 1]\cap\bQ$. By \Cref{L:type_determines_isomorphism_class}, the denominators of the elements of $\type(\cR_i)$ in reduced fraction form are uniquely determined by the isotropy groups ocurring in the oriented bridge $\sR_i$, Hence there are finitely many possibilities for the collections of types $\sigma$, and we have a finite disjoint union decomposition $f^{-1}(x) = \sqcup_{\sigma}f^{-1}(x)_\sigma$ into constructible subsets. We may replace $f^{-1}(x)$ with a fixed $f^{-1}(x)_\sigma$; we need to show that its dimension is bounded by $\ell$. By \Cref{L:type_determines_isomorphism_class}(1), for any geometric point the restriction $\cL^|_{\widetilde{\sC}^{st}}$ is isomorphic to the ideal sheaf $\cO(-{\mathscr{D}})$ of an effective Cartier divisor $\mathscr{D} \subset \widetilde{\sC}^{st}$ which is uniquely determined by $\sigma$. Furthermore, the isomorphism classes of the restrictions $\cL|_{\sR_i}$ are also uniquely determined $\sigma$ by \Cref{L:type_determines_isomorphism_class}(2). The possible choices of gluing data for obtaining $\cL$ in terms of $\cO(-\mathscr{D})$ and the line bundles $\cL|_{\cR^{(i)}}$ are parameterized by $\mathbb{G}_m^{2\ell} / \mathbb{G}_m^{\ell} \times \mathbb{G}_m^{r} \cong \mathbb{G}_m^{\ell - r}$, where $r$ is the number of connected components of $\widetilde{\mathscr{C}}^{\st}$. Furthermore, by \Cref{L:type_determines_isomorphism_class}(1), the cosection is given as the composition
     \[ \cL \twoheadrightarrow \cL|_{\widetilde{\mathscr{C}}^{\st}} \xrightarrow{\psi} \cO_{\widetilde{\sC}^{st}}(-\mathscr{D}) \hookrightarrow \cO_{\mathscr{C}}\]
    for some choice of isomorphism $\psi: \cL|_{\widetilde{\mathscr{C}}^{\st}} \xrightarrow{\sim} \cO_{\widetilde{\sC}^{st}}(-\mathscr{D})$ of line bundles. Hence for any fixed pair $(\mathscr{C}, \cL)$ the space of possible cosections is parameterized by $\mathbb{G}_m^r$. We conclude that the dimension of each of the finitely many $f^{-1}(x)_{\sigma}$ is at most $\ell -r + r = \ell$, as desired.
    
    \medskip 
    
    \noindent (2). This follows from flatness of $\st: \pertgen_g \to \overline{\cM}_g$, the fact that $\st$ is an isomorphism over the substack of smooth curves $\cM_g \subset \overline{\cM}_g$, and the integrality of $\overline{\cM}_g$.
\end{proof}

\subsection{The proof of {\Cref{T:perturbed_curves_main}}}

The proof will appear at the end of this subsection.%, after establishing the key geometric construction and its consequences.

%\Cref{L:rational_model_type_calculation} has several consequences for how types behave in families.

% $\phi : \sC \to C^{\st}$ is its stabilization, and $n \in C^{\st}(R)$ is a node that persists over the whole family. By choosing one of the preimages of $n$ in the normalization of $C^{\st}$ to be $p_-$, one orients the

The next lemma applies to a $\bG_m$-equivariant family $(\sC,\cL,s)$ in $\bigpert_g$ over $Y$, where $Y$ is one of the two schemes associated in \Cref{notn: Y over the stack} to a DVR $R$. Since $\overline{\cM}_{g}$ is separated with finite inertia, the composition $\cY = Y/\bG_m \to \bigpert_g \to \overline{\cM}_{g}$ factors through $\cY \to \Spec(R)$ by \cite[Lem. 3.6, Lem. 3.45]{alper2023existence}. Concretely, this means the stabilization of $\sC$ is pulled back from a family $C^{\st} \to \Spec(R)$ in the sense that there is a morphism $\sC \to C^{\st}$ such that $\sC \to Y \times_{\Spec(R)} C^{\st}$ is the stabilization.

As in \Cref{notn: semistable-curves-over-Y}, we use the notation $(u, v)$ for the pair of parameters cutting out $\mathfrak{o}\in Y$. We write $\eta_u$ and $\eta_v$ for the generic points of $Y_u$ and $Y_v$.

\begin{lemma}\label{L:pert_completeness}
     If the restriction to $Y^\circ$ is a family in $\pertgen_g$, then so is the whole family over $Y/\bG_m$. Furthermore, let $\sR_u, \sR_v$ and $\sR_{\mathfrak{o}}$ denote the rational bridges lying over a node $x \in C_{0}^{\st}$ in each of the three fibers over $\eta_u, \eta_v$ and $\mathfrak{o}$. Then a rational model for $C^{\st}$ at $x$ canonically orients the three bridges, and
    \[\type(\sR_{\mathfrak{o}}) = \type(\sR_{u}) \cup \type(\sR_{v}).\]
\end{lemma}
\begin{proof}
    First, we observe that the cosection $s$ on the fiber over $\mathfrak{o}$ can only vanish on rational bridges. This is because the vanishing locus of $s$ on the family $\sC \to Y$ has codimension at most one at every point, so if it meets a non-bridge point in the $\mathfrak{o}$ fiber it must also meet a non-bridge point in some other fiber, which would contradict \Cref{L:perturbed_curve_basic_properties}.

    Choose a rational model around $\Spec(R) \to \bA^1$ and $\varphi : \widehat{C^{\st}} \cong \widehat{P_R}$ around a node $x \in C^{\st}_{0}$. The condition that $s$ only vanishes on rational bridges implies the condition of \Cref{L:rational_model_type_calculation}, that $V(s) \setminus \pi^{-1}(Z)$ is closed. Therefore \Cref{L:rational_model_type_calculation} provides a family of genus $0$ twisted curves $\sP \to Y$, which by construction is $\bG_m$-equivariant because $\sC$ is. For any $y \in Y$, if $\cC_y \cap \pi^{-1}(Z)$ is non-empty, then it is a rational bridge $\sR_y \subset \sC_y$, and the orientation of $P^{\rm{un}}_{0}$ both orients $\sR_y$ and identifies a distinguished component $\sP_y^- \subset \sP_y$. The condition that $\ord_+(\sR_y) + \ord_-(\sR_y) = \deg(\cL|_{\sR_y})$ is equivalent to the condition that $\deg(\cL_\sP|_{\sP_y}) = 0$, so if it holds at one fiber of $Y$ it holds at every fiber.
    
    The condition that $\ord_\pm(\sR_y) \leq 1$ is equivalent to the condition that $\deg(\cL_\sP|_{\sP_y^{\pm}}) \geq -1$. Any connected proper subcurve $\sS \subset \sP_y$ that contains $\sP_y^-$ can be characterized as the smallest connected subchain that contains both $\sP^-_{\mathfrak{o}}$ and $x$ for some uniquely determined node $x \in \sP_{\mathfrak{o}}$. $x$ must deform to a node in one of the other fibers, so without loss of generality let us assume $x' \in \sP_{u}$ is a node that specializes to $x$. Then the original subcurve $\sS$ is the flat limit of the smallest connected subcurve $\sS' \subset \sP_{u}$ that contains $\sP_{u}^-$ and $x'$. Therefore, the set of degrees of subchains of this form in $\sP^-_{\mathfrak{o}}$ is the union of the sets of degrees of subchains of this form in $\sP^-_{u}$ and $\sP^-_{v}$. This simultaneously implies the upper bound on $\ord_\pm(\sR_y)$ and the formula for $\type(\sR_{\mathfrak{o}})$.
\end{proof}

\begin{proof}[Proof of \Cref{T:perturbed_curves_main}]
    By \Cref{prop: perturbed-curves-open}, $\pertgen^\cB_g \subset \bigpert_g$ is open for all $\cB$ as in \Cref{H:stability_condition}. Since $\bigpert_g$ is $\Theta$- and $S$-complete by \Cref{T:s-complete-and-theta-red}, \Cref{L:pert_completeness} shows that $\pertgen^\cB_g$ is $\Theta$- and $S$-complete as well.
    
    If $\cB$ is finite, \Cref{L:type_determines_isomorphism_class} implies that $\pertgen^\cB_g$ is bounded, and hence it admits a separated good moduli space by \cite[Thm. A]{alper2023existence}, which we call $X$. Because $\overline{\cM}_g \subset \pertgen^\cB_g$, the induced composition on moduli spaces $\overline{M}_g \to X \to \overline{M}_g$ is the identity, hence $\overline{M}_g \to X$ is a closed immersion. By the integrality of $\pertgen^\cB_g$ proven in \Cref{prop: pert curves stack is integral}, the open immersion $\overline{\cM}_g \hookrightarrow \pertgen^{\cB}_g$ is dense, and hence $\overline{M}_g \to X$ is also dense. This forces the closed immersion $\overline{M}_g \to X$ to be an isomorphism, and thus the same holds for $X \to \overline{M}_g$.
\end{proof}

\appendix

\section{Formal gluing of algebraic stacks}
\label[appendix]{appendix}

Following the terminology of \cite{hall2023mayer}, we say that a morphism of algebraic stacks $g : \cX' \to \cX$ and an open substack $\cU \subset \cX$ form a \emph{weak Mayer--Vietoris square} if $T \times_{\cX} \cX' \to T$ is an isomorphism for any map $T \to \cX$ whose image lies outside of $\cU$. We say the square is \emph{Tor-independent} if in addition the fiber product $T \times_{\cX} \cX'$ above is Tor-independent, and we say it is \emph{flat} if $g$ is flat. If both $\cX$ and $\cX'$ are noetherian, then Tor-independence is automatic by \cite[Lem 3.3]{hall2023mayer}.

Flat Mayer--Vietoris squares are stable under base change along an arbitrary map $\cV \to \cX$, but Tor-independent squares are only stable under base change along a morphism that is $g$-flat, in the sense that the fiber product $\cV \times_\cX \cX'$ is Tor-independent \cite[Lem.~3.1]{hall2023mayer}. Note that if $g$ is flat, then any morphism $\cV \to \cX$ is $g$-flat.

%composition with $g$ is an equivalence between the groupoid of maps $T \to \cX$ such that $T \times_\cX \cU = \emptyset$ and maps $T \to \cX'$ such that $T \times_{\cX'} g^{-1}(\cU) = \emptyset$. We call $\cU' := g^{-1}(\cU)$.

For any stack $\cS$, we let $\AlgSt(\cS)$ denote the $2$-category of algebraic stacks $\cY$ with a quasi-compact morphism $\cY \to \cS$ with affine relative diagonal, and we let $\AlgSt_{\rm{fp}}(\cS) \subset \AlgSt(\cS)$ denote the full subcategory of those $\cY \to \cS$ that are also of finite presentation.

\begin{thm}\label{T:formal_glueing}
    If $(g : \cX' \to \cX, \cU \subset \cX)$ is a flat Mayer-Vietoris square and $\cX$ is locally covered by the spectrum of $G$-rings, then the restriction morphism is an equivalence of $2$-categories
    \[
    \Phi : \AlgSt(\cX) \xrightarrow{\cong} \AlgSt(\cX') \times_{\AlgSt(\cU')} \AlgSt(\cU).
    \]
    It also restricts to an equivalence on full subcategories $\AlgSt_{\rm{fp}}(-) \subset \AlgSt(-)$ of morphisms that are in addition finitely presented.
\end{thm}

%We actually show the full faithfulness of $\Phi$ under the weaker hypothesis that $g$ is a ``Tor-independent Mayer-Vietoris square'' \cite[Def.~1.3]{hall2023mayer}, meaning for any morphism $T \to \cX$ whose image lies in $\cX \setminus \cU$, the fiber product $T \times_\cX \cX'$ is Tor-independent and $T \times_\cX \cX' \to T$ is an isomorphism.

\begin{lemma}\label{L:appendix_ff}
    If $(g : \cX' \to \cX, \cU \subset \cX)$ is a Tor-independent Mayer-Vietoris square, then for two stacks $\cV$ and $\cW$ over $\cX$, where $\cV \to \cX$ is $g$-flat, the canonical map
    \begin{equation}\label{E:appendix_ff}
        \Map_{\cX}(\cV,\cW) \to \Map_{\cX}(\cV_{\cX'},\cW) \times_{\Map_{\cX}(\cV_{\cU'},\cW)} \Map_{\cX} (\cV_{\cU}, \cW)
    \end{equation}
    is an equivalence of groupoids. In particular, the functor $\Phi$ is fully faithful when $g$ is flat.
\end{lemma}
\begin{proof}
%The presheaf of $2$-categories $\AlgSt(-)$ satisfies smooth descent, so it suffices to prove claim when $\cX$ is affine, which we will assume throughout.

By \cite[Lem.~3.1(2)]{hall2023mayer}, the square obtained by base change to $\cV$ is Tor-independent. Applying \cite[Prop.~7.1]{hall2023mayer} to this square gives full faithfulness of \eqref{E:appendix_ff}. Both sides of \eqref{E:appendix_ff} satisfy smooth descent in $\cV$, so it suffices to assume that $\cV$ is an affine scheme and show that for any morphism of presheaves over $\cX$, $\alpha : \cV_{\cX'} \cup_{\cV_{\cU'}} \cV_{\cU} \to \cW$, there is a smooth cover $V \to \cV$ such that the composition $V_{\cX'} \cup_{V_{\cU'}} V_{\cU} \to \cW$ extends to a morphism $V \to \cW$.

Consider a smooth surjection $W \to \cW$, where $W$ is affine. Because $\cW$ has affine diagonal, the morphism $W \to \cW$ is affine, so $V_{\cX'} := \cV_{\cX'} \times_{\cW} W \to \cV_{\cX'}$, $V_{\cU'} := \cV_{\cU'} \times_{\cW} W \to \cV_{\cU'}$, and $V_{\cU} := \cV_{\cU} \times_{\cW} W \to \cV_{\cU}$ are smooth surjective affine morphisms. It follows from \cite[Cor.~4.9]{hall2023mayer} that these glue to a morphism $V \to \cV$ that is $g$-flat, affine, and of finite presentation. Furthermore, $V \to \cV$ is flat by \cite[Thm.~4.4(3b)]{hall2023mayer}, and then one can deduce it is smooth and surjective because its fibers are smooth and nonempty. \cite[Thm.~4.4]{hall2023mayer} also implies that $\Gamma(V,\cO_V) \cong \Gamma(V_{\cX'},\cO_{V_{\cX'}}) \times_{\Gamma(V_{\cU'},\cO_{V_{\cU'}})} \Gamma(V_{\cU},\cO_{V_{\cU}})$, which implies that $V_{\cX'} \cup_{V_{\cU'}} V_{\cU} \to W$ extends to $V$.
\end{proof}

To simplify notation, for any $\cV$ over $\cX$, we let \[\sD(\cV) := \sG(\cX'_{\cV}) \times_{\sG(\cU'_{\cV})} \sG(\cU_{\cV}),\] and we define $\sD_{\mathrm{fp}}(-)$ in the same way using $\AlgSt_{\rm{fp}}(-)$. Objects of $\sD(\cV)$ are triples $(\cY',\cY_{\cU}, \alpha)$ where $\cY'$ is a stack over $\cX'_\cV$, $\cY_\cU$ is a stack over $\cU_\cV$, and $\alpha : \cY'_{\cU'_\cV} \cong (\cY_\cU)_{\cU'_\cV}$ is an isomorphism. By \Cref{L:appendix_ff}, if $g$ is flat then $\Phi_{\cV} : \sG(\cV) \to \sD(\cV)$ is fully faithful for any $\cV$, and we say that a triple $(\cY',\cY_\cU,\alpha)$ is realizable if it lies in the essential image of $\Phi_\cV$.

\begin{lemma}\label{L:appendix_limit}
    Suppose $g$ is flat and $\cV$ is a stack over $\cX$. Let $S_\infty = \lim_{\lambda} S_\lambda$ be a cofiltered inverse limit of affine $\cV$-schemes, and $(\cY',\cY_\cU,\alpha) \in \sD(\cV)$. If $(\cY',\cY_\cU,\alpha)|_{S_\infty}$ is realizable, then so is $(\cY',\cY_\cU,\alpha)|_{S_\lambda}$ for all sufficiently large $\lambda$.
\end{lemma}
\begin{proof}
    Let $\cY_\infty \to S_\infty$ be the stack that realizes the triple $(\cY',\cY_\cU,\alpha)|_{S_\infty}$. By absolute noetherian approximation (see \cite[Cor. C]{rydh2023absolutenoetherianapproximationalgebraic} and \cite[Thm. C(i)]{rydh_noetherian_app}), it factors as an affine morphism $\cY_\infty \to \cP_\infty$ followed by a morphism $\cP_\infty \to S_\infty$ of finite presentation with affine diagonal. For $\lambda$ sufficiently large, the morphism $\cP_\infty \to S_\infty$ descends to a morphism $\cP_\lambda \to S_\lambda$ of finite presentation with affine diagonal \cite[Props. B.2(ii), B.3(xix)]{rydh_noetherian_app}. The composition $\cY'|_{S_\infty} \to \cY_\infty \to \cP_\infty$ is the base change of a morphism $\cY'|_{S_\lambda} \to (\cP_\lambda)_{\cX'}$ for all sufficiently large $\lambda$, which can also be assumed to be affine (see \cite[Prop B.2(i)]{rydh_noetherian_app} and \cite[Thm. C(i)]{rydh_noetherian_app}). Likewise, the morphism $\cY_\cU|_{S_\infty} \to (\cP_\infty)_{\cU}$ and the compatibility $2$-isomorphism $\alpha$ descends to $S_\lambda$ for all sufficiently large $\lambda$. In other words the morphism $(\cY'|_{S_\infty},\cY_\cU|_{S_\infty},\alpha) \to (\cP_\infty|_{\cX'},\cP_\infty|_{\cU}, \rm{canonical})$ in $\sD(S_\infty)$ is pulled back from a morphism
    \[
    (\cY'|_{S_\lambda}, \cY_\cU|_{S_\lambda}, \alpha|_{S_\lambda}) \to (\cP_\lambda|_{\cX'}, \cP_\lambda|_{\cU}, \rm{canonical})
    \]
    in $\sD(S_\lambda)$ for all sufficiently large $\lambda$, that is affine on each of the two stacks. Affine gluing \cite[Cor.~4.10]{hall2023mayer} implies that this compatible pair of affine morphisms is realized by an affine morphism $\cY_\lambda \to \cP_\lambda$ for sufficiently large $\lambda$. $\cY_\lambda$ realizes the triple $(\cY'|_{S_\lambda}, \cY_\cU|_{S_\lambda}, \alpha|_{S_\lambda})$ over $S_\lambda$.
\end{proof}

\begin{lemma}\label{L:appendix_essential_surjectivity_section}
    If $g$ is flat and admits a section $s : \cX \to \cX'$, then $\cY := s^\ast(\cY')$ is a realization for any $(\cY',\cY_\cU,\alpha) \in \cD(\cX)$.
\end{lemma}
\begin{proof}
The proof of \cite[Thm.~B(4)]{hall2023mayer} applies here, replacing \cite[Cor. 7.2]{hall2023mayer} by \Cref{L:appendix_ff}. The reversed square $(s : \cX \to \cX', \cU' \subset \cX')$ is Tor-independent by \cite[Lem.~3.5]{hall2023mayer}. Pulling back $\alpha$ along $s|_{\cU}:\cU \to \cU'$ gives an identification $\beta: \cY|_{\cU} \cong \cY_\cU$.

Pullback along a flat morphism is Tor-independent of any section of that morphism, so $g^\ast(\cY)$ is $s$-flat, and $\cY'|_{\cU'} \cong g_\cU^\ast (\cY_\cU)$ is $s_\cU$-flat, so \cite[Cor.~4.8]{hall2023mayer} applied to smooth affine charts of $\cY'$ implies that $\cY'$ is $s$-flat. The stacks $\cY'$ and $g^\ast(\cY)$ over $\cX'$ give isomorphic triples with respect to the reversed square $(s : \cX \to \cX', \cU')$, and both are $s$-flat, so \Cref{L:appendix_ff} implies that we have an isomorphism $g^\ast(\cY) \cong \cY'$ that is compatible with the isomorphism of triples. This in turn, can be used to show that $\cY$ is a realization for the original triple over $\cX$.
\end{proof}

\begin{proof}[Proof of \Cref{T:formal_glueing}] The presheaf of $2$-categories $\sG(-)$ satisfies smooth descent, and \Cref{L:appendix_ff} shows that $\Phi_\cV$ is fully faithful for any $\cV$. To show $\Phi_\cV$ is essentially surjective, it suffices to show that for any $\cY \in \sD(\cV)$, there is a smooth surjection $V \to \cV$ such that $\cY|_{V}$ lies in the image of $\Phi_V$. Furthermore, it suffices to show this when $\cV = \cX = \Spec(A)$ is affine, with $A$ a $G$-ring.

We use the quasi-separated argument from the proof of \cite[Thm.~B(4)]{hall2023mayer}, with \Cref{L:appendix_essential_surjectivity_section} and \Cref{L:appendix_limit} in place of its section and finite-presentation limit steps:

Fix $(\cY',\cY_\cU,\alpha) \in \sD(\cX)$. For $x \in \cX \setminus \cU$, let $C = \widehat{\cO_{\cX,x}}$. The Mayer–Vietoris property gives compatible sections of $\cX'_C \to \Spec(C)$ over $\Spec(C/\mathfrak{m}^{n+1})$ for every $n$. Since $\cX'_C$ is affine, taking the inverse limit gives an actual section over $\Spec(C)$. Therefore \Cref{L:appendix_essential_surjectivity_section} shows that the triple is realizable after base change along $\Spec(C) \to \cX$. The map $A \to C$ is regular because $A$ is a $G$-ring, so by Popescu’s theorem \cite[\href{https://stacks.math.columbia.edu/tag/07GC}{Tag 07GC}]{stacks-project}, $C \cong \colim C_\lambda$ is a filtered colimit with $C_\lambda$ smooth and finitely presented over $A$. \Cref{L:appendix_limit} shows that the realization over $C$ descends to $C_\lambda$ for some $\lambda$. Repeating this for finitely many $x$ and including $\cU$ gives a smooth cover of $\cX$ such that the restriction of the original triple is realizable.
\end{proof}

Given an algebraic stack $\cX$ and closed subset $Z \subset |\cX|$, let us say that a \emph{modification} of $\cX$ over $Z$ is a quasi-compact morphism $f : \cY \to \cX$ with affine relative diagonal such that the restriction $f^{-1}(\cX \setminus Z) \to \cX \setminus Z$ is an isomorphism.

\begin{prop}\label{P:nisnevich_pushout}
    Suppose $\cX$ is an algebraic stack, $Z \subset |\cX|$ is a closed subset, and $(g : \cX' \to \cX, \cX \setminus Z)$ is a flat Mayer-Vietoris square. Then base change along $g$ is an equivalence between the $2$-category of modifications of $\cX$ over $Z$ and the $2$-category of modifications of $\cX'$ over $g^{-1}(Z)$. Furthermore, a modification $\cY \to \cX$ is separated if and only if $\cY|_{\cX'} \to \cX'$ is.
\end{prop}
\begin{proof}
    A modification of $\cX$ over $Z$ is an object $\cY \in \AlgSt(\cX)$ where $\cY|_{\cU} \cong \cU$. Since $\cU$ is final among stacks over $\cU$, under the equivalence $\Phi$ of \Cref{T:formal_glueing}, these correspond to the triples $(\cY', \cU, \alpha)$, ie. with modifications $\cY'$ of $\cX'$ over $g^{-1}(Z)$ lying in $\AlgSt(\cX')$. The first claim then follows from \Cref{T:formal_glueing}.

    For the separatedness, if $\cY \to \cX$ is separated, then so is the base change $\cY|_{\cX'} \to \cX'$. Conversely, smooth-locally on $\cX$, the morphism $\cX' \sqcup \cU \to \cX$ is an fpqc covering. The diagonal of $\cY \to \cX$ is affine, so $\cY \to \cX$ is separated if and only if its diagonal is finite, which can be checked after base change to $\cX' \sqcup \cU$.
\end{proof}

\bibliographystyle{alpha}
\bibliography{biblio}
\end{document}